\documentclass[11pt,reqno]{amsart}
\usepackage{amsmath,amssymb,latexsym,esint,cite,mathrsfs,dsfont}
\usepackage{verbatim,cite,relsize,color,soul}
\usepackage[latin1]{inputenc}
\usepackage{microtype}
\usepackage{color,enumitem,graphicx}
\usepackage[colorlinks=true,urlcolor=blue, citecolor=red,linkcolor=blue,
linktocpage,pdfpagelabels, bookmarksnumbered,bookmarksopen]{hyperref}
\usepackage[hyperpageref]{backref}
\usepackage[english]{babel}
\usepackage{tikz}
\usepackage[font=small,skip=5pt]{caption}
\usepackage{amsrefs}

\usepackage{geometry}
\numberwithin{equation}{section}

\newtheorem{theorem}{Theorem}[section]
\newtheorem{lemma}[theorem]{Lemma}
\newtheorem{proposition}[theorem]{Proposition}
\newtheorem{corollary}[theorem]{Corollary}

\newtheorem{claim}[theorem]{Claim}

\newtheoremstyle{remarkstyle}
{}{}{\itshape}{ }{\bfseries}{.}{ }{\thmname{#1}\thmnumber{ #2}\thmnote{ (#3)}}
\theoremstyle{remarkstyle}
\newtheorem{remark}{Remark}[section]

\newcommand{\R}{\mathbb R}
\newcommand{\C}{\mathbb C}
\newcommand{\Ac}{\mathcal A}
\newcommand{\Bc}{\mathcal B}
\newcommand{\Rc}{\mathcal R}

\newcommand{\Sc}{\mathcal S}
\newcommand{\Vc}{\mathcal V}
\newcommand{\Sb}{\mathbb S}

\newcommand{\varep}{\varepsilon}
\DeclareMathOperator*{\loc}{loc}

\DeclareMathOperator*{\opt}{opt}
\DeclareMathOperator*{\dip}{dip}

\DeclareMathOperator*{\supp}{supp}
\DeclareMathOperator*{\sr}{sr}
\DeclareMathOperator*{\ur}{ur}
\DeclareMathOperator*{\rea}{Re}
\DeclareMathOperator*{\ima}{Im}
\DeclareMathOperator*{\deltc}{{\delta_c}}
\DeclareMathOperator*{\csb}{{c}}
\DeclareMathOperator*{\nls}{{DNLS}}

\newcommand{\wihat}[1]{\widehat{#1}}

\title[Threshold Dynamics for  Dipolar Quantum Gases]
{Mass-Energy Threshold Dynamics for  Dipolar Quantum Gases}

\author[V. D. Dinh]{Van Duong Dinh}
\address[V. D. Dinh]{Laboratoire Paul Painlev\'e UMR 8524, Universit\'e de Lille CNRS, 59655 Villeneuve d'Ascq 
and 
Department of Mathematics, HCMC University of Education, 280 An Duong Vuong, Ho Chi Minh, Vietnam}
\email{contact@duongdinh.com}

\author[L. Forcella]{Luigi Forcella}
\address[L. Forcella]{\'Ecole Polytechnique F\'ed\'erale de Lausanne, Institute of Mathematics, Station 8, CH-1015 Lausanne, Switzerland}
\email{luigi.forcella@epfl.ch}

\author[H. Hajaiej]{Hichem Hajaiej}
\address[H. Hajaiej]{Department of Mathematics, California State University, Los Angeles, CA 90032}
\email{hhajaie@calstatela.edu}


\begin{document}
	
	\begin{abstract}
		We consider a Gross-Pitaevskii equation which appears as a model in the description of  dipolar Bose-Einstein condensates, without a confining external trapping potential. We describe the asymptotic dynamics of solutions to the corresponding Cauchy problem in the energy space in different configurations with respect to the mass-energy threshold, namely for initial data above and at the mass-energy threshold. We first establish a scattering criterion for the equation that we prove by means of the concentration/compactness and rigidity scheme. This criterion enables us to show the energy scattering for solutions with data above the mass-energy threshold,  for which only blow-up was known. We also prove a blow-up/grow-up criterion for the equation with general data in the energy space. As a byproduct of scattering and blow-up criteria, and the compactness of minimizing sequences for the Gagliardo-Nirenberg's inequality, we study long time dynamics of solutions with data lying exactly at the mass-energy threshold.
	\end{abstract}
	
	\maketitle
	
	\section{Introduction}
	\label{S1}
	\setcounter{equation}{0}
	\subsection{Introduction}
	In this paper, we consider the Cauchy problem associated to the following non-local  non-linear Schr\"odinger equation arising as a model to describe a dipolar Bose-Einstein Condensate (BEC) at low temperatures:
	\begin{equation}
	\label{eq:evolution}
	i \hbar \frac{\partial u}{\partial t} = - \frac{\hbar^2}{2m}\Delta u + W(x) u + U_0|u|^2 u + (V_{\dip}\ast |u|^2) u, 
	\end{equation}
	where the wave function $u(t,x)$ is a complex function $u:\R\times\R^3\to\C$, $t\in\R$ is the time variable, and  $x$ denotes the space variable in the three-dimensional Euclidean space $\R^3$. The parameters $\hbar$ and $m$ involved in Gross-Pitaevskii equation (GPE) \eqref{eq:evolution} are the Planck constant and the mass of a dipolar particle, respectively. The coefficient  $U_0 := 4 \pi \hbar^2 a_s /m$ describes the local interaction between dipoles in the condensate, which is defined in terms of the $s$-wave scattering length $a_s$. The dipolar interaction kernel $V_{\dip}(x)$ is defined, up to some physical constants (in particular the vacuum magnetic permeability and the permanent magnetic dipole moment), in the following way:
	\[
	V_{\dip}(x)=\frac{1-3\cos^2\theta}{|x|^3},
	\]
	where $\theta$ is the angle between the dipole axis and $x.$ By calling the unitary dipole axis $n$, the previous means that
	\[
	\theta=\arccos\left(\frac{x\cdot n}{|x|}\right).
	\]
	The potential $W(x)$  is a time-independent real-valued function representing an external harmonic confinement that we assume to be zero in the current work, which is equivalent to the physical situation where the external potential is turned off.
	
	
	We refer the reader to \cite{LMSLP, PS, YY1, YY2} for an in-depth justification of the previous GPE as a model for dipolar quantum gases. Let us also mention the papers \cite{AEM, BSTH, DMA} for an overview about the first experimental observations of  Bose-Einstein condensation after which there has been an always increasing interest in the theoretical investigation of such models, as well as numerical studies (we mention \cite{BaCa, BaCaWa, HMS} and reference therein  for numerical simulations).
	
	For a rigorous mathematical study of \eqref{eq:evolution}, it is opportune to consider the latter in its dimensionless formulation, see \cite{BJM}. We will essentially focus on its associated Cauchy problem:
	
	\begin{align} \label{dip-NLS}
	\left\{
	\renewcommand*{\arraystretch}{1.2}
	\begin{array}{rcl}
	i\partial_t u + \frac{1}{2}\Delta u &=& \lambda_1 |u|^2 u + \lambda_2 (K\ast |u|^2) u, \quad (t,x)\in [0,\infty)\times \R^3, \\
	 u(0,x)  &=& u_0(x) \in H^1(\R^3),
	\end{array}
	\right.	
	\end{align}
	where $\lambda_1, \lambda_2 \in \R$ and 
	\[
	K(x)= \frac{x_1^2 + x_2^2 - 2x_3^2}{|x|^5}.
	\]
	
	\noindent	It is worth mentioning that the expression defining the dipolar interaction kernel $K(x)$ is exactly  $V_{\dip}(x)$ given above with the choice of the dipole axis to be $n=(0,0,1).$  The coefficients  $\lambda_1$ and $\lambda_2$ are instead two real parameters defined in terms of the physical constants appearing in \eqref{eq:evolution}, which measure the strength of the two nonlinear terms, i.e. the local and the non-local one.
	
We can partition the $(\lambda_1, \lambda_2)$-coordinate plane in the following two subsets called  \emph{unstable regime} and \emph{stable regime}, respectively: 
	\begin{align}  \label{uns-reg}
	\Rc_{\ur} := \left\{ (\lambda_1, \lambda_2) \in \R^2  \left| 
	\renewcommand*{\arraystretch}{1.3}
	\begin{array}{lcl}
	\lambda_1 <\frac{4\pi}{3} \lambda_2 &\text{if}& \lambda_2>0 \\
	\lambda_1 <-\frac{8\pi}{3} \lambda_2 & \text{if} & \lambda_2<0
	\end{array}
	\right.
	\right\},
	\end{align}
	
	\begin{align} \label{sta-reg}
	\Rc_{\sr} := \left\{ (\lambda_1, \lambda_2) \in \R^2 \left| 
	\renewcommand*{\arraystretch}{1.3}
	\begin{array}{lcl}
	\lambda_1 \geq \frac{4\pi}{3} \lambda_2 &\text{if}& \lambda_2>0 \\
	\lambda_1 \geq -\frac{8\pi}{3} \lambda_2 & \text{if} & \lambda_2<0
	\end{array}
	\right.
	\right\}.
	\end{align}
This terminology has been introduced in the pioneering paper of Carles, Markowich, and Sparber \cite{CMS}, where the first mathematical treatment for the Gross-Pitaevskii equation governing a dipolar BEC has been undertaken.

	
\begin{remark}
Condition \eqref{sta-reg} guarantees that the energy  of a solution to \eqref{dip-NLS} is non-negative, and as showed in \cite{BF}, in the stable regime all solutions are global and scatter. With respect to the latter fact, we can naively think that the stable regime plays the same role of the defocusing character for the cubic NLS equation. Nevertheless, it is only  condition \eqref{cond-GW} below (which is a subset of the unstable regime)  which guarantees instead that the potential energy of \eqref{dip-NLS} is non-positive.  Moreover, as proved in \cite{CMS}, provided that $0<\lambda_1<\frac{4\pi}{3}\lambda_2,$ namely for a defocusing character of the  local term and a strictly positive coupling coefficient, collapse in finite time may occur provided that the energy is negative.
So it is improper to refer to \eqref{uns-reg} and \eqref{sta-reg} as giving the defocusing/focusing correspondence of \eqref{dip-NLS} with respect to the usual cubic NLS.  \\
Concerning the asymptotic dynamics, let us point-out that the main feature of \eqref{dip-NLS} subject to the conditions \eqref{uns-reg} is the existence of standing waves. See below, in the Introduction, for further discussions.
\end{remark}

	Solutions to \eqref{dip-NLS} enjoy the conservation along the flow of mass, defined as
	\begin{equation*}
	M(u(t)) = \|u(t)\|^2_{L^2(\R^3)} = M(u_0), 
	\end{equation*}
	and energy, defined by 
	\begin{equation*}
	E(u(t)) = \frac{1}{2} \left(H(u(t)) + N(u(t))\right) = E(u_0),
	\end{equation*}
	where
	\begin{equation}\label{defi-H}
	H(f):= \|\nabla f\|^2_{L^2(\R^3)}
	\end{equation}
	and
	\begin{equation}\label{defi-N}
	N(f):= \int_{\R^3} \lambda_1 |f(x)|^4 + \lambda_2 (K\ast |f(x)|^2)|f(x)|^2 dx 
	\end{equation}
	are, up to some constant, the kinetic and the potential energy, respectively, associated to \eqref{dip-NLS}. 
	The previous two conservation laws can be formally proved by using  integration by parts; the rigorous derivation can be justified by an approximation argument. 
	
	It was shown in \cite[Lemma 2.1]{CMS} that the dipolar kernel $K(x)$ defines, through a convolution, a continuous linear map from  $L^p(\R^3)$  into itself, for any fixed $p\in(1,\infty),$ i.e. $	f\mapsto K\ast f$ satisfies $\|K\ast f\|_{L^p}\leq C\|f\|_{L^p}$. Moreover, it was proved in  \cite[Lemma 2.3]{CMS}, by using the decomposition of $e^{-ix\cdot\xi}$ into spherical harmonics, that   
	
	\begin{align} \label{fourier-trans-K}
	\wihat{K}(\xi) = \frac{4\pi}{3} \left(\frac{2\xi_3^2-\xi_1^2 -\xi_2^2}{|\xi|^2} \right) \in \left[-\frac{4\pi}{3}, \frac{8\pi}{3}\right],
	\end{align}
	where $\mathcal{F}(f)(\xi)=\wihat f(\xi)=\mathlarger{\int}_{\R^3}e^{-ix\cdot\xi}f(x) dx$ is the Fourier transform of $f(x).$
	Thanks to the Plancherel's identity, we can therefore express the potential energy $N(f)$ as
	\[
	N(f) = (2\pi)^{-3} \int_{\R^3} \left( \lambda_1+\lambda_2 \wihat{K}(\xi)\right) \left|\widehat{|f|^2}(\xi)\right|^2 d\xi.
	\]
	The equation \eqref{dip-NLS} is invariant under the scaling
	\begin{align} \label{scaling}
	u_\mu(t,x):= \mu u(\mu^2 t, \mu x), \quad \mu >0.
	\end{align}
	A straightforward calculation shows that 
	\[
	\|u_\mu(0)\|_{\dot{H}^\gamma(\R^3)} = \mu^{\gamma -\frac{1}{2}} \|u_0\|_{\dot{H}^\gamma(\R^3)}
	\]
	which in turn implies that  under the  scaling \eqref{scaling} the  $\dot{H}^{\frac{1}{2}}(\R^3)$-norm of an initial datum is preserved.

	As mentioned above, in the unstable regime the equation \eqref{dip-NLS} admits standing waves, i.e. solutions of the form $u(t,x) = e^{it} \phi(x)$, where $\phi \in H^1(\R^3)$ is a non-trivial solution to the elliptic equation
	\begin{align} \label{ell-equ}
	-\frac{1}{2} \Delta \phi + \phi + \lambda_1 |\phi|^2 \phi +\lambda_2 (K\ast |\phi|^2) \phi=0.
	\end{align}
	The first result about standing waves for \eqref{dip-NLS} is due to Antonelli and Sparber in  \cite{AS}, where the authors proved their existence by showing the existence of minimizers for the Weinstein functional (see \cite{Weinstein})
	\begin{align} \label{weins-func}
	W(f):= \frac{(H(f))^{\frac{3}{2}} (M(f))^{\frac{1}{2}} }{-N(f)}
	\end{align}
	over the set
	\[
	\Bc:= \left\{ f\in H^1(\R^3) \ : \ N(f) <0\right\}.
	\]
	They also established qualitative properties such as symmetry, regularity, and decay for real positive $H^1(\R^3)$-solutions to \eqref{ell-equ}. 
	
	A second approach to show existence of standing states, which consists on minimizing the energy functional under prescribed $L^2(\R^3)$-norm, is due  to Bellazzini and Jeanjean in \cite{BJ}, and relies on topological methods (see also \cite{CH}). See also the recent paper \cite{Dinh-2021} for the existence of standing waves via the minimization of the action functional.

	In this paper, we are interested in the long time dynamics including global existence, energy scattering, finite time blow-up or grow-up in infinite time, and  concentration phenomena  of solutions to \eqref{dip-NLS} in the unstable regime. It is known that  long time dynamics for solutions to \eqref{dip-NLS} is closely related to the existence of ground states. Here by ground state we mean a non-trivial solution to \eqref{ell-equ} which minimizes the Weinstein functional \eqref{weins-func} over all functions $f \in \Bc.$
	
As the aim of our paper is to give results about the asymptotic of solutions to \eqref{dip-NLS} in the unstable regime,  we recall some known results about the behaviour of solutions under the conditions \eqref{uns-reg}, and then we proceed by stating our main contributions. 
	
	\begin{remark}
		 In the stable regime, it has been proved in \cite{BF} that for any $H^1(\R^3)$ initial datum the corresponding solution to \eqref{dip-NLS} is global and scatters.  Hence we will only focus on the unstable regime given by  \eqref{uns-reg}.
	\end{remark}

	As mentioned before, the local well-posedness of \eqref{dip-NLS} was proved in \cite{CMS}, and among other results, the existence of finite time blow-up for \eqref{dip-NLS} was shown for initial data with negative energy.  
	
	 For the reader's convenience, let us start by reviewing some important results in the unstable regime.
	
	\begin{proposition}[\cite{CMS}] \label{theo-blow-CMS}
		Let $\lambda_1$ and $\lambda_2$ satisfy \eqref{uns-reg}. Let $u_0 \in \Sigma:= H^1(\R^3) \cap L^2(\R^3;|x|^2 dx)$ satisfy $E(u_0)<0$. Then the corresponding solution \eqref{dip-NLS} blows-up in finite time.
	\end{proposition}
	The proof of this result is based on an argument of Glassey \cite{Glassey} using virial identities related to \eqref{dip-NLS}, see \cite[Theorem 5.2]{CMS}. Let us mention the fact that in the finite-variance space $\Sigma,$ the result in \cite[Theorem 5.2]{CMS} is true  in the presence of a harmonic confinement as well, for energies bounded by a (positive) constant  depending on the smallest trap-frequency and the $L^2(\R^3;|x|^2 dx)$-norm of the initial datum.\\
	
	In \cite{BF}, Bellazzini and the second author established the scattering for $H^1(\R^3)$-solutions to \eqref{dip-NLS} below the ground state threshold in the unstable regime. 	
	
	\begin{theorem}[\cite{BF}]\label{theo-scat-BF}
		Let $\lambda_1$ and $\lambda_2$ satisfy \eqref{uns-reg}. Let $u_0 \in H^1(\R^3)$ satisfy $E(u_0) M(u_0) < E(\phi) M(\phi)$ and $H(u_0) M(u_0) <H(\phi) M(\phi)$, where $\phi$ is a ground state related to \eqref{ell-equ}. Then the corresponding solution $u(t)$ to \eqref{dip-NLS} exists globally in time and scatters in $H^1(\R^3)$ in both directions, that is, there exist $u^\pm \in H^1(\R^3)$ such that
		\begin{align} \label{defi-scat}
		\lim_{t\rightarrow \pm \infty} \|u(t)- L(t) u^\pm\|_{H^1(\R^3)} =0,
		\end{align}
		where $L(t):= e^{it\frac{1}{2}\Delta}$ is the Schr\"odinger free propagator.
	\end{theorem}
	The proof of the Theorem above is done by employing a concentration/compactness and rigidity argument in the spirit of Kenig and Merle \cite{KM}. 	In Remark \ref{rem-indep} we point-out how the mass-energy thresholds above are independent of the ground state, which is indeed not unique.\\
		It is worth mention that the statement of the previous theorem was actually given in \cite{BF} by means of the boundedness of the initial energy in terms of the mountain pass energy level of the initial datum and the positivity of the Pohozaev functional. More precisely, for any fixed mass $c>0,$ one defines
		\[
		\Sc(c):= \left\{ f \in H^1(\R^3) \ : \ M(f)=c\right\}.
		\]
It was proved in \cite{BJ} that for $c>0$, the energy functional has a mountain pass geometry on $S(c)$. 
Moreover, the energy level $\gamma(c)$ has the following variational characterization 
		\[
		\gamma(c) = \inf \left\{E(f) \ : \ f \in \Vc(c)\right\}, \quad \Vc(c):= \left\{ f\in H^1(\R^3) \ : \ M(f) =c, \,G(f)=0 \right\},
		\]
where
		\begin{equation}\label{def:G}
G(f):= H(f) +\frac{3}{2} N(f).
		\end{equation}
Note that the functional $G$ is exactly the virial functional associated to \eqref{dip-NLS}, namely $G(u(t))= \frac{d^2}{dt^2} V(t),$ where
	
\begin{align} \label{defi-V}
	V(t):= \int_{\R^3}|x|^2|u(t,x)|^2 dx
	\end{align}
 is the variance at time $t$ of the mass density.\\
 	
The conditions on the initial datum leading to global well-posedness and scattering in Theorem \ref{theo-scat-BF} as in the original paper \cite{BF} are $E(u_0)<\gamma\left(\|u_0\|_{L^2(\R^3)}^2\right)$ and $G(u_0)>0.$

		\noindent Let us denote 
		\begin{align} 
		\Ac^+&:= \left\{ f \in H^1(\R^3) \ |  \ E(f) < \gamma(c), c = M(f), \,G(f)>0\right\}, \label{defi-A+}\\
		\Ac^-&:= \left\{ f \in H^1(\R^3)  \ | \ E(f) < \gamma(c), c = M(f), \, G(f)<0\right\}, \label{defi-A-}
		\end{align}
and
		\begin{align} 
		\tilde{\Ac}^+&:= \left\{ f \in H^1(\R^3) \left| \begin{array} {ccc}
		E(f) M(f) &<& E(\phi) M(\phi) \\
		H(f) M(f) &<& H(\phi) M(\phi)
		\end{array} \right. \right\}, \label{defi-tilde-A+}\\
		\tilde{\Ac}^-&:= \left\{ f \in H^1(\R^3) \left| \begin{array} {ccc}
		E(f) M(f) &<& E(\phi) M(\phi) \\
		H(f) M(f) &>& H(\phi) M(\phi)
		\end{array} \right. \right\}. \label{defi-tilde-A-}
		\end{align}

It was shown in \cite[Proposition 3.2]{BF-blow} that $\Ac^+ \equiv \tilde{\Ac}^+$ and $\Ac^- \equiv \tilde{\Ac}^-$. Therefore the statement of Theorem \ref{theo-scat-BF} is completely equivalent to the original one in \cite{BF}. Furthermore,  in \cite{BF-blow}  the finite time blow-up is showed for \eqref{dip-NLS} in the unstable regime \eqref{uns-reg}, by assuming that the initial datum $u_0 $ belongs to $ \Ac^-,$ is cylindrical symmetric, and has finite variance in $x_3$-direction.\\ 

Recently, Gao and Wang \cite{GW} showed a finite time blow-up result for \eqref{dip-NLS} with arbitrarily large energy in a subset of the unstable regime. 
	\begin{theorem}[\cite{GW}] \label{theo-blow-GW}
		Let $\lambda_1$ and $\lambda_2$ satisfy 
		\begin{align} \label{cond-GW}
		\left\{
		\renewcommand*{\arraystretch}{1.3}
		\begin{array}{lcl}
		\lambda_1 <-\frac{8\pi}{3} \lambda_2 &\text{if}& \lambda_2>0, \\
		\lambda_1 <\frac{4\pi}{3} \lambda_2 & \text{if} & \lambda_2<0.
		\end{array}
		\right.
		\end{align} 
		Let $\phi$ be a ground state related to \eqref{ell-equ}. Let $u_0 \in \Sigma$ satisfy 
		\begin{align}
		\frac{E(u_0)M(u_0)}{E(\phi)M(\phi)} &\left(1-\frac{(V'(0))^2}{8E(u_0)V(0)}\right) \leq 1, \label{cond-1-GW} \\
		-N(u_0)M(u_0) &> -N(\phi)M(\phi), \label{cond-2-GW} \\
		V'(0) &\leq 0. \label{cond-3-GW}
		\end{align}
		Then the corresponding solution $u(t)$ to \eqref{dip-NLS} blows-up forward in finite time.
	\end{theorem}
	The proof of the latter result, see \cite[Theorem 2]{GW} is based on the argument of Duyckaerts and Roudenko \cite{DR-beyond} using virial identities. Note that the restriction \eqref{cond-GW} ensures that the potential energy  takes negative values (see \cite[Lemma 4]{GW} and \cite{GW2} for some applications to the Hartree equation). 
	
	\subsection{Main results}
	We are ready to state our main results, and we start with the following  scattering criterion for \eqref{dip-NLS} in the unstable regime. The following theorem will play an essential role in the proof of the dynamics above the mass-energy threshold (see Theorem \ref{theo-scat-above} below). 
	\begin{theorem} \label{theo-scat-crite}
		Let $\lambda_1$ and $\lambda_2$ satisfy \eqref{cond-GW}. Let $\phi$ be a ground state related to \eqref{ell-equ}. Let $u(t)$ be a $H^1(\R^3)$-solution to \eqref{dip-NLS} defined on the maximal forward time interval $[0,T^*)$. Assume that
		\begin{align} \label{scat-crite}
		\sup_{t\in [0,T^*)} -N(u(t)) M(u(t)) < -N(\phi) M(\phi).
		\end{align}
		Then $T^*=\infty$ and the solution $u(t)$ scatters in $H^1(\R^3)$ forward in time.
	\end{theorem}
	
	The main strategy in the proof of this Theorem is a concentration/compactness and rigidity scheme in the same spirit of \cite{HR} (see also \cite{FXC} for a generalization to the all range of inter-critical powers). The main difference comparing our result to the one of Holmer and Roudenko in \cite{HR}, is that in the latter paper data below the mass-energy ground state are considered. This allows them to use the construction of wave operators to show that nonlinear profiles associated to the linear ones exist globally in time and have finite scattering norm. In our setting, such a construction does not work. To overcome the difficulty, we prove a Pythagorean expansion along bounded non-linear flows (see the proof of the existence of a critical solution, and in particular see Lemma \ref{lem-pytha-along-inls}, later on in the paper).
		
	
	\begin{remark}
		As mentioned above, condition \eqref{cond-GW} guarantees that the potential energy is always negative along the time evolution of \eqref{dip-NLS} which is needed in the proof (see Remark \ref{rem-cond-GW}.)
	\end{remark}
	
	\begin{remark}\label{rem-indep}
		Although the existence of ground states related to \eqref{ell-equ} is proved (see e.g. \cite{AS}), the uniqueness (up to symmetries) of positive ground states related to \eqref{ell-equ} is not known. However we point-out that the quantities
		\begin{align} \label{inde-quant}
		E(\phi)M(\phi), \quad H(\phi)M(\phi), \quad -N(\phi)M(\phi)
		\end{align}
		do not depend on the choice of the ground state $\phi$ (see \eqref{inde-quant-proof}).
	\end{remark}
	
	As a consequence of Theorem \ref{theo-scat-crite}, we can give the energy scattering result for \eqref{dip-NLS} above the mass-energy threshold, which is a complementary result of the one of Gao and Wang \cite{GW}, as the latter only addressed formation of singularities in finite time. 
	\begin{theorem} \label{theo-scat-above}
		Let $\lambda_1$ and $\lambda_2$ satisfy \eqref{cond-GW}. Let $\phi$ be a ground state related to \eqref{ell-equ}, and $u_0 \in \Sigma$ be such that
		\begin{align} 
		E(u_0)M(u_0) &\geq  E(\phi) M(\phi), \label{cond-1-above} \\
		\frac{E(u_0)M(u_0)}{E(\phi)M(\phi)}&\left(1-\frac{(V'(0))^2}{8 E(u_0)V(0)}\right) \leq 1, \label{cond-2-above} \\
		-N(u_0) M(u_0) &< -N(\phi) M(\phi),\label{cond-scat-1-above} \\
		V'(0) &\geq 0, \label{cond-scat-2-above}
		\end{align}
		Then the corresponding solution $u(t)$ to \eqref{dip-NLS} satisfies \eqref{scat-crite}. In particular, the solution exists globally and scatters in $H^1(\R^3)$ forward in time.
	\end{theorem}

	Our next result is the following blow-up or grow-up result for \eqref{dip-NLS} in the unstable regime.
	\begin{theorem} \label{theo-blow-crite}
		Let $\lambda_1$ and $\lambda_2$ satisfy \eqref{uns-reg}. Let $u(t)$ be a $H^1(\R^3)$-solution to \eqref{dip-NLS} defined on the maximal forward time interval $[0,T^*)$. Assume that
		\begin{align} \label{blow-crite}
		\sup_{t\in [0,T^*)} G(u(t)) \leq -\delta
		\end{align}
		for some $\delta>0$. Then either $T^*<\infty$, or $T^*=\infty$ and there exists a sequence of time $t_n \rightarrow \infty$ such that
		\[
		\|\nabla u(t_n)\|_{L^2(\R^3)}\rightarrow \infty 
		\]
		as $n\rightarrow \infty.$  In the latter case we say that the solution grows-up. In particular, if $u_0$ has finite variance, i.e. $V(0)<\infty$, then $T^*<\infty,$ namely the solution blows-up in finite time.
	\end{theorem}
	
	The strategy to prove the previous theorem is in the spirit of Du, Wu, and Zhang, see \cite{DWZ}, but here we need to control the non-local term involved in \eqref{dip-NLS}.
	We point-out that the hypothesis \eqref{blow-crite} is actually non-empty. Indeed, Bellazzini and the second author  proved in \cite{BF-blow} that \eqref{blow-crite} is always satisfied provided the initial datum belongs to \eqref{defi-tilde-A-}, or equivalently to \eqref{defi-A-}, namely below the mass-energy threshold.
	
\begin{corollary} 
Let $\lambda_1$ and $\lambda_2$ satisfy \eqref{uns-reg}, and $\phi$ be a ground state related to \eqref{ell-equ}. Assume that $u_0 \in H^1(\R^3)$ is such that
	\begin{equation}
	\left\{
		\begin{aligned} 
		E(u_0)M(u_0) &< E(\phi) M(\phi), \\
		H(u_0) M(u_0)&> H(\phi) M(\phi),
		\end{aligned}
	\right.
	\end{equation}

\noindent and let $u(t)$ the solution to \eqref{dip-NLS}. Then either $T^*<\infty$, or $T^*=\infty$ and there exists a sequence of time $t_n \rightarrow \infty$ such that
		\[
		\|\nabla u(t_n)\|_{L^2(\R^3)}\rightarrow \infty 
		\]
as $n\to\infty.$ If $V(0)<\infty,$  blow-up always occurs in finite time.	
\end{corollary}	
	
	Our last result addresses  the long time dynamics for \eqref{dip-NLS} at the mass-energy threshold.
	
	\begin{theorem}\label{theo-dyna-at}
		Let $\lambda_1$ and $\lambda_2$ satisfy \eqref{uns-reg}. Let $\phi$ be a ground state related to \eqref{ell-equ}. Let $u_0 \in H^1(\R^3)$ be such that
		\begin{align} 
		E(u_0)M(u_0) = E(\phi) M(\phi). \label{cond-ener-at}
		\end{align}
		\begin{itemize}
			\item If 
			\begin{align} \label{cond-scat-at}
			H(u_0) M(u_0) < H(\phi) M(\phi),
			\end{align}
			then the corresponding solution $u(t)$ to \eqref{dip-NLS} satisfies
			\begin{align} \label{est-solu-at-1}
			H(u(t)) M(u(t)) < H(\phi) M(\phi)
			\end{align}
			for all $t$ in the existence time. In particular, the solution exists globally in time. Moreover, if we assume in addition that $\lambda_1$ and $\lambda_2$ satisfy \eqref{cond-GW}, then the solution either scatters in $H^1(\R^3)$ forward in time, or there exist a time sequence of times $t_n\rightarrow \infty$, a ground state $\tilde{\phi}$ related to \eqref{ell-equ}, and a sequence $\{y_n\}_{n\geq1}\subset\R^3$ such that
			\begin{equation} \label{conver-tn}
			u(t_n, \cdot-y_n) \rightarrow e^{i\theta} \mu \tilde{\phi}(\mu\cdot) \quad \text{ strongly in } H^1(\R^3)
			\end{equation}
			as $n\rightarrow \infty$, for some $\theta \in \R$ and $\mu>0$.
			\item If
			\begin{equation} \label{cond-at}
			H(u_0) M(u_0)= H(\phi) M(\phi),
			\end{equation}
			then there exists a ground state $\tilde{\phi}$ related to \eqref{ell-equ} such that the solution $u(t)$ to \eqref{dip-NLS} satisfies $u(t,x) = e^{i\mu^2t} e^{i\theta} \mu \tilde{\phi}(\mu x)$ for some $\theta\in \R$ and $\mu>0$.
			\item If
			\begin{align} \label{cond-blow-at}
			H(u_0) M(u_0)> H(\phi) M(\phi),
			\end{align}
			then the corresponding solution $u(t)$ to \eqref{dip-NLS} satisfies
			\begin{align} \label{est-solu-at-3}
			H(u(t)) M(u(t)) > H(\phi) M(\phi)
			\end{align}
			for all $t$ in the lifespan of the solution. Furthermore, the solution either blows-up forward in finite time,
			\begin{itemize}
				\item[i)] or there exists $t_n\rightarrow \infty$ such that $\|\nabla u(t_n)\|_{L^2(\R^3)} \rightarrow \infty$ as $n\rightarrow \infty$;
				\item[ii)] or there exists $t_n \rightarrow \infty$ such that \eqref{conver-tn} holds for some sequence $\{y_n\}_{n\geq1}\subset\R^3,$  $\theta\in \R,$ and  $\mu>0.$
			\end{itemize}
\noindent In addition, if $V(0)<\infty$, then the possibility depicted in $i)$ is excluded. 
		\end{itemize}
	\end{theorem}
\begin{remark} To the best of our knowledge, the first work addressing the long time dynamics for the focusing cubic NLS at the ground state threshold is due to Duyckaerts and Roudenko \cite{DR-threshold}. Their proof is based on delicate spectral properties of the ground state for the cubic NLS. Recently, the first author in \cite{Dinh-DCDS} considered the dynamics for the NLS at the threshold for the whole range of inter-critical powers. 
\end{remark}

\noindent 	The proof of the latter theorem is based on the scattering and blow-up/grow-up criteria given respectively in Theorems \ref{theo-scat-crite} and \ref{theo-blow-crite}, and the compactness of minimizing sequence for the Gagliardo-Nirenberg-type inequality related to \eqref{dip-NLS} (see Lemma \ref{lem-compact}). We also take the advantage of the scaling invariance \eqref{scaling}. For more details, we refer to Section \ref{S5}.
	
	\begin{remark}
Throughout the paper, we do not consider the dynamics for radial solutions to \eqref{dip-NLS} since the dipole-nonlinearity vanishes when applied to radially symmetric  functions, and therefore \eqref{dip-NLS} would reduce to the classical NLS equation. It is due to the fact the average of the dipolar kernel $K(x)$ vanishes on $\Sb^2$ (see \cite{CMS}).  Hence this symmetry hypothesis won't be useful to obtain other weaker sufficient conditions for the global dynamics of solutions to \eqref{dip-NLS}.
	\end{remark}	
\begin{remark}
All the results stated above can be proved for negative times in similar fashions.  
\end{remark}	
	
\subsection{Organization of the paper} The paper is organized as follows. In Section \ref{S2}, we give some preliminary results including the local well-posedness, the small data scattering, and the stability result for \eqref{dip-NLS}, and some properties of ground states related to \eqref{ell-equ}. In Section \ref{S3}, we give the proof of the scattering criterion given in Theorem \ref{theo-scat-crite} and its application to the energy scattering above the mass-energy threshold given in Theorem \ref{theo-scat-above}. Section \ref{S4} is devoted to the proof of the blow-up/grow-up result given in Theorem \ref{theo-blow-crite}. Finally, in Section \ref{S5}, we study the dynamics of solutions to \eqref{dip-NLS} with data lying exactly at the mass-energy threshold given in Theorem \ref{theo-dyna-at}.
	
	\section{Preliminary results and notation}
	\label{S2}
	\setcounter{equation}{0}
	In the next subsection, we recall some basic facts on local theory for \eqref{dip-NLS}. We recall that $L(t):= e^{it\frac{1}{2}\Delta}$ stands for the linear Schr\"odinger propagator, namely $v(t,x)=L(t)u_0(x):= e^{it\frac{1}{2}\Delta} u_0(x)$ solves $i\partial_t v+\frac12\Delta v=0$ with $v(0,x)=u_0(x).$ In what follows, given an interval $I\subseteq \mathbb R,$ bounded or unbounded, we denote by $L^p(I,L^q)$ the Bochner space of vector-valued functions $f:I\mapsto L^q(\R^3)$ endowed with the usual norm $$\|f\|_{L^p(I,L^q)}=\left(\int_{I}\|f(s)\|_{L^q(\R^3)}^p ds\right)^{1/p}, \quad p\in[1,\infty).$$ 
	For $p=\infty$ we adopt the usual modification. Here $L^q=L^q(\R^3)$ are the usual Lebesgue spaces. For an initial datum $u_0$ in $H^1(\R^3)$ satisfying $V(0)<\infty,$ we will also use the notation $u_0\in\Sigma:=H^1(\R^3)\cap L^2(\R^3, |x|^2dx).$
Since we will only work in the three-dimensional space, we will drop the notation $\R^3$ from now on.

	\subsection{Local theory} Let us start with the following nonlinear estimates. 
	
	\begin{lemma}\label{lem-non-est}
		There exists $C>0$ such that for any time interval $I \subset \R$, the following estimates hold
		\begin{align*}
		\||u|^2 u + (K\ast |u|^2) u\|_{L^{8/5}(I, L^{4/3})} &\leq C \|u\|^2_{L^8(I, L^4)} \|u\|_{L^{8/3}(I,L^4)}, \\
		\||u|^2 u + (K \ast |u|^2) u \|_{L^{8/3}(I,L^{4/3})} &\leq C \|u\|^3_{L^8(I,L^4)}, \\
		\|\nabla(|u|^2 u + (K\ast |u|^2) u)\|_{L^{8/5}(I, L^{4/3})} &\leq C \|u\|^2_{L^8(I, L^4)} \|\nabla u\|_{L^{8/3}(I,L^4)}.
		\end{align*}
	\end{lemma}
	For the proof of the above nonlinear estimates, we refer the reader to the proof of Proposition 3.4 in \cite{CMS}. Here we note that $(8/3,4)$ is a Schr\"odinger $L^2$-admissible pair and $(8/5, 4/3)$ is its dual pair. See the monographs \cite{Ca, LP, Tao} and references therein for a more detailed and complete treatment on the Strichartz estimates.
	
	Thanks to the above nonlinear estimates, Strichartz estimates, and the contraction mapping argument, we have the following local well-posedness in $H^1$ for \eqref{dip-NLS}, see \cite{CMS}.  
	
	\begin{lemma} \label{lem-lwp}
		Let $\lambda_1, \lambda_2 \in \R$ and $u_0 \in H^1$. Then there exist $T_*,T^*\in (0,\infty]$ and a unique solution to \eqref{dip-NLS} satisfying
		\[
		u \in C((-T_*,T^*), H^1) \cap L^{8/3}_{\loc}((-T_*,T^*), L^4).
		\]
		The solution satisfies the conservation laws of mass and energy, i.e. $M(u(t))= M(u_0)$ and $E(u(t))=E(u_0)$ for all $t\in (-T_*,T^*)$. Moreover, there is a blow-up alternative:  either the solution exists globally in time or  $T^*<\infty$ (resp. $T_*<\infty$) and 
		\[
		\lim_{t\nearrow T^*}\|\nabla u(t)\|_{L^2}=\infty \quad \left(\text{resp. }  \lim_{t\searrow -T_*} \|\nabla u(t)\|_{L^2} =\infty \right),
		\]
	
	\end{lemma}
	The next result gives  a sufficient condition for scattering,  see \cite{BF} for a proof. 
	\begin{lemma} \label{lem-scat-cond}
		Let $u:[0,\infty) \times \R^3 \rightarrow \C$ be a $H^1$-solution to \eqref{dip-NLS}. If $\|u\|_{L^\infty([0,\infty), H^1)} <\infty$ and $\|u\|_{L^8([0,\infty),L^4)} <\infty$, then the solution scatters in $H^1$ forward in time in the sense of \eqref{defi-scat}.
	\end{lemma}
	Note that the $L^8L^4$ space is invariant under the scaling \eqref{scaling}. See also the original paper by Cazenave and Weissler \cite{CW} in the context of the so-called \emph{rapidly decaying solution} for NLS. It is also worth mentioning that $(8,4)$ is a $\dot H^{1/2}$ Strichartz admissible pair. The following provides a small data scattering result for \eqref{dip-NLS}.
	\begin{lemma}\label{lem-small-data}
		There exists $\delta>0$ such that, provided
		\[
		\|L(t) u_0\|_{L^8(\R,L^4)} \leq \delta,
		\]
		where $u_0\in H^1,$
		then the corresponding solution to \eqref{dip-NLS} exists globally in time and satisfies
		\begin{align*}
		\|u\|_{L^8(\R,L^4)} &\leq 2 \|L(t) u_0\|_{L^8(\R,L^4)}, \\
		\|u\|_{L^\infty(\R, H^1)} + \|u\|_{L^{8/3}(\R, W^{1,4})} &\leq C\|u_0\|_{H^1},
		\end{align*}
		for some constant $C>0$. In particular, the solution scatters in $H^1$ in both directions.		
	\end{lemma}
	In \cite[Lemma 3.2]{BF}, the small data scattering was stated with small $H^1$-norm of initial data.  Lemma \ref{lem-small-data} can be proved with a small refinement of the argument in \cite[Lemma 3.2]{BF}. We omit the details.
	
	Using Lemma \ref{lem-non-est} and an argument similar to \cite[Proposition 2.3]{HR}, we have the following stability result for \eqref{dip-NLS}. See also \cite[Section 3.7]{Tao} and reference therein for this kind of classical results.
	\begin{lemma}\label{lem-stability}
		Let $0 \in I \subseteq \R$ and $\tilde{u}: I \times \R^3 \rightarrow \C$ be a solution to 
		\[
		i \partial_t \tilde{u} + \frac{1}{2} \Delta \tilde{u} - F(\tilde{u}) = e, \quad 
		\]
		where 
		\begin{align} \label{defi-F}
		F(\tilde{u}):= \lambda_1 |\tilde{u}|^2 \tilde{u} + \lambda_2 (K\ast |\tilde{u}|^2) \tilde{u}.
		\end{align}
		Assume that $u(0,x)  = \tilde{u}_0(x)$ and
		\[
		\|\tilde{u}\|_{L^\infty(I,H^1)} \leq M_1, \quad \|\tilde{u}\|_{L^8(\R,L^4) \cap L^\infty(I,L^3)} \leq M_2
		\]
		for some constants $M_1,M_2>0$. Let $u_0 \in H^1$ be such that
		\[
		\|u_0-\tilde{u}_0\|_{H^1} \leq M_3, \quad \|L(t)(u_0- \tilde{u}_0)\|_{L^8(\R,L^4) \cap L^\infty(I,L^3)} \leq \varep
		\]
		for some $M_3>0$ and some $0<\varep<\varep_1=\varep_1(M_1,M_2,M_3)$. Suppose that
		\[
		\|e\|_{L^{8/5}(I,W^{1,4/3})} + \|e\|_{L^{8/3}(I,L^{4/3})} \leq \varep.
		\]
		Then there exists a unique solution $u:I\times \R^3\rightarrow \C$ to \eqref{dip-NLS} with initial data $u(0,x) = u_0(x)$ satisfying
		\begin{align*}
		\|u- \tilde{u}\|_{L^8(\R,L^4) \cap L^\infty(I,L^3)} &\leq C(M_1,M_2,M_3) \varep, \\
		\|u\|_{L^{8/3}(I,W^{1,4})} + \|u\|_{L^8(\R,L^4) \cap L^\infty(I,L^3)} &\leq C(M_1,M_2,M_3).
		\end{align*}
	\end{lemma}
	
	\subsection{Variational analysis}
	In this subsection, we recall some basic properties of ground states related to \eqref{ell-equ}.
	
	Let $\phi$ be a ground state related to \eqref{ell-equ}. Since it is an minimizer for the Weinstein functional \eqref{weins-func}, we have
	\[
	C_{\opt} = \frac{-N(\phi)}{ (H(\phi))^{\frac{3}{2}} (M(\phi))^{\frac{1}{2}}},
	\]
	where $C_{\opt}$ is the optimal constant in the Gagliardo-Nirenberg-type inequality
	\begin{align} \label{GN-ineq}
	-N(f) \leq C_{\opt} (H(f))^{\frac{3}{2}} (M(f))^{\frac{1}{2}}, \quad f \in H^1(\R^3).
	\end{align}
	We also have the following Pohozaev's identities (see \cite[Lemma 2.2]{AS}):
	\begin{align} \label{poho-iden}
	H(\phi) = 6 M(\phi)=-\frac{3}{2} N(\phi).
	\end{align}
	We infer that
	\[
	E(\phi) = \frac{1}{6}H(\phi) = -\frac{1}{4}N(\phi)
	\]
	and
	\begin{align} \label{opti-cons}
	C_{\opt} = \frac{2}{3} \left(H(\phi)M(\phi)\right)^{-\frac{1}{2}}.
	\end{align}
	This shows that
	\begin{align} \label{inde-quant-proof}
	E(\phi) M(\phi) = \frac{1}{6} H(\phi) M(\phi) = -\frac{1}{4} N(\phi) M(\phi) = \frac{2}{27} (C_{\opt})^{-2}.
	\end{align}
	In particular, we see that the quantities $E(\phi) M(\phi), H(\phi) M(\phi)$, and $N(\phi)M(\phi)$ are independent of $\phi$.

As a consequence of the variational analysis above, we give the following coercivity property. 	
	
	\begin{lemma} \label{lem-bound-below}
		Let $\lambda_1$ and $\lambda_2$ satisfy \eqref{uns-reg}. Let $\phi$ be a ground state related to \eqref{ell-equ}, and $f \in H^1$ satisfy
		\[
		-N(f) M(f) \leq A <- N(\phi) M(\phi)
		\]
		for some constant $A>0$. Then there exists $\nu=\nu(A,\phi)>0$ such that
		\begin{align} 
		G(f) &\geq \nu H(f), \label{bound-G} \\
		E(f) &\geq \frac\nu2H(f).\label{bound-E}
		\end{align}
	\end{lemma}
	
	\begin{proof}
		If $N(f) \geq 0$, then it is obvious, from \eqref{def:G} and the definition of the energy that $G(f) \geq H(f)$ and $E(f) \geq \frac{1}{2} H(f)$. If $N(f)<0$, then we write
		\[
		A = -(1-\eta) N(\phi) M(\phi)
		\]
		for some $\eta=\rho(A,\phi) \in (0,1)$, and we have
		\[
		-N(f)M(f) \leq -(1-\eta)N(\phi)M(\phi).
		\]
		Using \eqref{GN-ineq}, \eqref{opti-cons} and \eqref{inde-quant-proof}, we see that
		\begin{align*}
		(-N(f))^{\frac{3}{2}} &\leq C_{\opt}  (-N(f) M(f))^{\frac{1}{2}} (H(f))^{\frac{3}{2}} \\
		&= \left(\frac{-N(f)M(f)}{-N(\phi) M(\phi)}\right)^{\frac{1}{2}} \left(\frac{2}{3}H(f) \right)^{\frac{3}{2}}\\
		&\leq (1-\eta)^{\frac{1}{2}} \left(\frac{2}{3}H(f) \right)^{\frac{3}{2}},
		\end{align*}
		which implies 
		\[
		\frac{2}{3} (1-\eta)^{\frac{1}{3}} H(f) \geq -N(f).
		\]
		It follows that
		\[
		G(f) = H(f) +\frac{3}{2}N(f) \geq \left(1- (1-\eta)^{\frac{1}{3}}\right) H(f)
		\]
		which proves \eqref{bound-G}. The estimate \eqref{bound-E} follows from \eqref{bound-G} and 
		\[
		E(f)=\frac{1}{2} G(f) -\frac{1}{4} N(f) \geq \frac{1}{2} G(f).
		\]
		The proof is complete.
	\end{proof}

	\section{Dynamics above the threshold}\label{S3}

	This section is devoted to the proof of the scattering criterion given in Theorem \ref{theo-scat-crite} and its consequences.
	
\subsection{Proof of the scattering criterion}	
Let $u: [0,T^*) \times \R^3 \rightarrow \C$ be a $H^1$-solution to \eqref{dip-NLS} satisfying \eqref{scat-crite}. By the conservation of energy and \eqref{scat-crite}, we see that $\sup_{t\in [0,T^*)} H(u(t)) \leq C(E, \phi) <\infty$. Hence by the blow-up alternative we have $T^*=\infty$. \\
	
	Let $A$ and $\delta$ be two positive real numbers. For $u(t)$ satisfying 
 	\begin{align} \label{cond-A-delta}
	\sup_{t\in [0,\infty)} -N(u(t)) M(u(t)) \leq A, \quad E(u)M(u) \leq \delta.
	\end{align}
we define 
	\begin{align} \label{S-A-delta}
	S(A,\delta):= \sup \left\{ \|u\|_{L^8([0,\infty),L^4)} \ : \ u(t) \text{ is a } H^1 \text{ solution to } \eqref{dip-NLS} \text{ satisfying } \eqref{cond-A-delta}\right\}.
	\end{align}
	\\
		We see that Theorem $\ref{theo-scat-crite}$ is reduced to show the following result.
	\begin{proposition} \label{prop-scat-equi}
		Let $\lambda_1$ and $\lambda_2$ satisfy \eqref{cond-GW}. Let $\phi$ be a ground state related to \eqref{ell-equ}. If $A<-N(\phi)M(\phi)$, then for all  $\delta >0$, $S(A,\delta)<\infty$.
	\end{proposition}  
\begin{remark} We remark that the energy $E(u)$ is non-negative due to Lemma \ref{lem-bound-below}.
\end{remark}	
	
	\noindent {\it Proof of Proposition \ref{prop-scat-equi}.}
	The proof of Proposition $\ref{prop-scat-equi}$ is done by several steps. In what follows, the quantity $A$ is fixed, and satisfies the boundedness assumption as in the statement of the Proposition.\\
	
	\noindent \textit{Step 1. Small data theory.} By  interpolation and Lemma \ref{lem-bound-below} we have
	\[
	\|u_0\|^4_{\dot{H}^{\frac{1}{2}}} \leq H(u_0)M(u_0) \leq \frac{2E(u_0)}{\nu}M(u_0) \leq \frac{2\delta}{\nu}.
	\]
	By taking $\delta>0$ sufficiently small, Strichartz estimates imply $\|L(t) u_0\|_{L^8(\R,L^4)} \ll 1$ which, by Lemma \ref{lem-small-data}, implies $S(A,\delta)<\infty$. \\
	
	\noindent  \textit{Step 2. Existence of a critical solution.} Assume by contradiction that $S(A,\delta) =\infty$ for some $\delta >0$. By Step 1, the quantity
	\begin{align} \label{defi-delta-c}
	\deltc=\delta_c(A):= \inf \left\{ \delta>0 \ : \ S(A,\delta) =\infty\right\}
	\end{align}
	is well-defined and positive. We infer from the definition of $\deltc$ that:
	\begin{itemize}
		\item[$\mathcal A)\label{case-A:page}$]\label{item-A} If $u(t)$ is a $H^1$-solution to \eqref{dip-NLS} which satisfies  
		\[
		\sup_{t\in [0,\infty)} -N(u(t))M(u(t)) \leq A, \quad E(u) M(u) <\deltc,
		\]
		then $\|u\|_{L^8([0,\infty),L^4)} <\infty$ and the solution scatters forward in time.
		\item[$\mathcal B)$]\label{item-B} There exists a sequence of $H^1$-solutions $\{u_n(t)\}_{n\geq1}$ to \eqref{dip-NLS} with initial data $\{u_{0,n}\}_{n\geq1}$ such that
		\begin{align} \label{assump-un}
		\begin{aligned}
		\sup_{t\in [0,\infty)} -N(u_n(t))M(u_n(t)) &\leq A \text{ for all } n, \\
		E(u_n) M(u_n) &\searrow \deltc \text{ as } n\rightarrow \infty, \\
		\|u_n\|_{L^8([0,\infty),L^4)} &=\infty \text{ for all } n.
		\end{aligned}
		\end{align}
	\end{itemize}
	We will prove that there exists a $H^1$-solution $u_{\csb}(t)$ to \eqref{dip-NLS}  
	such that
	\begin{align}
	\begin{aligned} \label{prop-uc}
	M(u_{\csb}) &=1, \\
	\sup_{t\in [0,\infty)} -N(u_{\csb}(t)) &\leq A, \\
	E(u_{\csb}) &= \deltc, \\
	\|u_{\csb}\|_{L^8([0,\infty),L^4)} &=\infty.
	\end{aligned}
	\end{align}
	To this aim, we consider the sequence of initial data $\{u_{0,n}\}_{n\geq1}$. Using the scaling \eqref{scaling}, we may assume that $M(u_{0,n})=1$ for all $n$. Note that this scaling does not affect $E(u_n)M(u_n)$ and $\sup_{t\in [0,\infty)} -N(u_n(t))M(u_n(t))$. After this scaling, we have
	\begin{align} \label{prop-un-0}
	\begin{aligned}
	M(u_{0,n}) &=1 \text{ for all } n, \\
	\sup_{t\in [0,\infty)} -N(u_n(t)) &\leq A \text{ for all } n, \\
	E(u_{0,n}) &\searrow \deltc \text{ as } n\rightarrow \infty, \\
	\|u_n\|_{L^8([0,\infty),L^4)} &=\infty \text{ for all } n.
	\end{aligned}
	\end{align}
	Applying the profile decomposition (see e.g. \cite{DHR}) to the sequence $\{u_{0,n}\}_{n\geq 1}$ (which is now uniformly bounded in $H^1$), we have that for each integer $J\geq 1$, there exists a subsequence, still denoted by $\{u_{0,n}\}_{n\geq 1}$, and for each $1\leq j \leq J$, there exist a  profile $\psi^j \in H^1$, a sequence of time shifts $\{t^j_n\}_{n\geq 1} \subset \R$,  a sequence of space shifts $\{x^j_n\}_{n\geq 1} \subset \R^3$, and a sequence of remainders $\{W^J_n\}_{n\geq 1} \subset H^1$ such that
	\begin{align} \label{prof-decom-un-0}
	u_{0,n}(x) = \sum_{j=1}^J L(-t^j_n) \psi^j(x-x^j_n) + W^J_n(x).
	\end{align}
	The time and space shifts have the pairwise divergence property below 
	\begin{align} \label{time-space-diver}
	\lim_{n\rightarrow \infty} |t^j_n-t^k_n| + |x^j_n - x^k_n| = \infty, \quad \hbox{for any}\quad 1\leq j \ne k \leq J.
	\end{align}
	The remainder has the following asymptotic smallness property
	\begin{align} \label{remainder}
	\lim_{J \rightarrow \infty} \left[ \lim_{n\rightarrow \infty} \|L(t) W^J_n\|_{L^8(\R, L^4)\cap L^4(\R, L^6)} \right] =0.
	\end{align}
	Note that $(8,4)$ and $(4,6)$ are both $\dot H^{1/2}$-admissible Strichartz pairs.	
	Moreover, for fixed $J$, we have the asymptotic Pythagorean expansions (see e.g. \cite{BF})
	\begin{align} 
	\|u_{0,n}\|^2_{\dot{H}^\gamma} &= \sum_{j=1}^J \|\psi^j\|^2_{\dot{H}^\gamma} + \|W^J_n\|^2_{\dot{H}^\gamma} + o_n(1), \quad \forall \gamma \in [0,1], \label{H-gam-expan}\\
	-N(u_{0,n}) &= -\sum_{j=1}^J N(L(-t^j_n) \psi^j)  - N(W^J_n) + o_n(1), \label{N-expan} \\
	E(u_{0,n}) &= \sum_{j=1}^J E(L(-t^j_n) \psi^j) + E(W^J_n) + o_n(1). \label{ener-expan}
	\end{align}
	Here we note that the functionals $H, N$ (see \eqref{defi-H} and \eqref{defi-N}), and $E$ are invariant under the spatial translation.	In addition, we may assume that either $t^j_n\equiv 0$ or $t^j_n \rightarrow \pm \infty$, and either $x^j_n\equiv 0$ or $|x^j_n| \rightarrow \infty$.
	
	Next, we define the nonlinear profile $v^j$ associated to $\psi^j$ and $t^j_n$ as follows:
	\begin{itemize}
		\item If $t^j_n \rightarrow -\infty$, then $v^j$ is the maximal lifespan solution to \eqref{dip-NLS} that scatters forward in time to $L(t)\psi^j$, i.e. $\|v^j(t) -L(t) \psi^j\|_{H^1} \rightarrow 0$ as $t \rightarrow \infty$. In particular, we have $\|v^j(-t^j_n) -L(-t^j_n) \psi^j\|_{H^1} \rightarrow 0$ as $n\rightarrow \infty$ and $\|v^j(t)\|_{L^8((T_0,\infty),L^4)}<\infty$, for some $T_0>0.$
		\item If $t^j_n \rightarrow \infty$, then $v^j$ is the maximal lifespan solution to \eqref{dip-NLS} that scatters backward in time to $L(t) \psi^j$. In particular, we have $\|v^j(-t^j_n) -L(-t^j_n) \psi^j\|_{H^1} \rightarrow 0$ as $n\rightarrow \infty$ and $\|v^j(t)\|_{L^8((-\infty,-T_0),L^4)}<\infty$, for some $T_0>0.$
		\item If $t^j_n \equiv 0$, then $v^j$ is the maximal lifespan solution to \eqref{dip-NLS} with data $v^j(0)=\psi^j$. 
	\end{itemize}
	For each $j, n\geq 1$, we introduce $v^j_n(t):= v^j(t-t^j_n)$. By the definition, we have
	\begin{align} \label{est-v-jn}
	\|v^j_n(0)- L(-t^j_n) \psi^j\|_{H^1} \rightarrow 0 \quad \text{as } n\rightarrow \infty.
	\end{align}
	We thus can rewrite \eqref{prof-decom-un-0} as
	\begin{align} \label{non-prof-decom-un-0}
	u_{0,n}(x) = \sum_{j=1}^J v^j_n(0,x-x^j_n) + \tilde{W}^J_n(x),
	\end{align}
	where 
	\[
	\tilde{W}^J_n(x) = \sum_{j=1}^J L(-t^j_n) \psi^j(x-x^j_n) - v^j_n(0,x-x^j_n) + W^J_n(x). 
	\]
	By Strichartz estimates, we have
	\begin{align*}
	\|L(t) \tilde{W}^J_n\|_{L^8(\R,L^4)\cap L^4(\R,L^6)} \lesssim \sum_{j=1}^J \|L(-t^j_n) \psi^j - v^j_n(0)\|_{H^1} + \|L(t) W^J_n\|_{L^8(\R,L^4)\cap L^4(\R,L^6)} 
	\end{align*}
	which, by \eqref{remainder} and \eqref{est-v-jn}, implies
	\begin{align} \label{remainder-W}
	\lim_{J\rightarrow \infty} \lim_{n\rightarrow \infty} \|L(t) \tilde{W}^J_n\|_{L^8(\R,L^4)\cap L^4(\R,L^6)} =0.
	\end{align}
	Similarly, we infer from \eqref{ener-expan} and \eqref{est-v-jn} that
	\begin{align} \label{ener-expan-vj}
	E(u_{0,n}) &= \sum_{j=1}^J E(v^j_n(0)) + E(\tilde{W}^J_n) + o_n(1).
	\end{align}
	We need to prove the following Pythagorean expansions along the bounded NLS flow (see e.g. \cite[Lemma 3.9]{Guevara} for a similar result in the context of classical NLS). In what follows  $\nls(t)f$ denotes the solution to \eqref{dip-NLS} with initial data $f$.
	\begin{lemma}\label{lem-pytha-along-inls}
		Let $T\in (0,\infty)$ be a fixed time. Assume that for all $n \geq 1$, $u_n(t) \equiv \nls(t)u_{0,n}$ exists up to time $T$ and
		\begin{align} \label{est-H-un-T}
		\lim_{n\rightarrow \infty} \sup_{t\in [0,T]} H(u_n(t)) <\infty.
		\end{align}
			Consider the profile decomposition \eqref{non-prof-decom-un-0}. Denote $\tilde{W}^J_n(t)= \nls(t)\tilde{W}^J_n$. Then for all $t\in [0,T]$,
		\begin{align} \label{expan-H-un-t}
		H(u_n(t)) = \sum_{j=1}^J H(v^j_n(t)) + H(\tilde{W}^J_n(t)) + o_{J,n}(1),
		\end{align}
		where $o_{J,n}(1)$ satisfies $\lim_{J\to\infty} \limsup_{n\to\infty} o_{J,n}(1)=0,$ 
uniformly on $0\leq t\leq T$.  In particular, by the conservation of energy and \eqref{ener-expan-vj}, we have for all $1\leq j \leq J$ and all $t\in [0,T]$,
		\begin{align} \label{expan-N-un-t}
		-N(u_n(t)) = -\sum_{j=1}^J N(v^j_n(t)) -N(\tilde{W}^J_n(t)) + o_{J,n}(1).
		\end{align}
	\end{lemma}
	
	\begin{proof}[Proof of Lemma \ref{lem-pytha-along-inls}]
		By the Pythagorean expansion \eqref{H-gam-expan}, there exists $J_0$ large enough such that $\|\psi^j\|_{H^1}$ are sufficiently small for all $j \geq J_0+1$. This, together with \eqref{est-v-jn}, implies that $\|v^j_n(0)\|_{H^1}$ are small for all $j \geq J_0+1$. By the small data theory, $v^j_n$ exists globally in time and scatters in $H^1$ in both directions for all $j \geq J_0+1$. 
		
	\noindent 	Next, we reorder the first $J_0$ profiles and let $0 \leq J_2 \leq J_0$ such that 
		\begin{itemize}
			\item For any $1\leq j \leq J_2$, the time shifts $t^j_n \equiv 0$ for any $n$. Here if $J_2=0$, then it means that there is no $j$'s in this case.
			\item For any $J_2 +1 \leq j \leq J_0$, the time shifts $|t^j_n|\rightarrow \infty$ as $n\rightarrow \infty$. Note that if $J_2=J_0$, then there is no $j$'s in this case.
		\end{itemize}
		Fix $T\in (0,\infty)$ and assume that $u_n(t) := \nls(t) u_{0,n}$ exist up to time $T$ and satisfy \eqref{est-H-un-T}. We observe that for $J_2+1 \leq j \leq J_0$, 
		\begin{align} \label{obser-vj-tn}
		\|v^j_n\|_{L^8([0,T],L^4)} \rightarrow 0
		\end{align} 
		as $n\rightarrow \infty$. Indeed, if $t^j_n\rightarrow \infty$, then as $\|v^j(t)\|_{L^8((-\infty,-T_0), L^4)} <\infty$ for some $T_0>0$, we have
		\[
		\|v^j_n\|_{L^8([0,T],L^4)} = \|v^j(t)\|_{L^8([-t^j_n,T-t^j_n],L^4)} \rightarrow 0
		\]
		as $n\rightarrow \infty$.  A similar claim is valid for $t^j_n\rightarrow -\infty.$
	Moreover,  for $J_2+1 \leq j \leq J_0$ and any $2<r\leq 6$ we get 
		\begin{align} \label{est-vj-Lr}
		\|v^j_n\|_{L^\infty([0,T], L^r)} \rightarrow 0
		\end{align}
		as $n\rightarrow \infty$. Indeed we have
		\begin{align*}
		\|v^j_n\|_{L^\infty([0,T],L^r)} &\leq \|L(t-t^j_n) \psi^j\|_{L^\infty([0,T],L^r)} + \|v^j_n(t) - L(t-t^j_n) \psi^j\|_{L^\infty([0,T],L^r)} \\ 
		& \leq \|L(t-t^j_n) \psi^j\|_{L^\infty([0,T],L^r)} + C \|v^j_n(t) -L(t-t^j_n) \psi^j\|_{L^\infty([0,T],H^1)}.
		\end{align*}
		By the decay of the linear flow, the first term tends to zero as $n$ tends to infinity. For the second term, we use the Duhamel formula
		\[
		v^j_n(t) = L(t) v^j_n(0) + i \int_0^t L(t-s) F(v^j_n(s)) ds
		\]
		with $F$ as in \eqref{defi-F} and Lemma \ref{lem-non-est}, to have
		\[
\|v^j_n\|_{L^{8/3}([0,T],W^{1,4})} \lesssim \|v^j_n(0)\|_{H^1} + \|v^j_n\|^2_{L^8([0,T], L^4)} \|v^j_n\|_{L^{8/3}([0,T],W^{1,4})}
\]
which, by \eqref{obser-vj-tn}, implies $\|v^j_n\|_{L^{8/3}([0,T],W^{1,4})}$ is bounded uniformly in $n.$ Similarly we have
		\begin{align*}
		\|v^j_n(t) &- L(t-t^j_n) \psi^j\|_{L^\infty([0,T],H^1)} \\ 
		&\lesssim \|L(t) v^j_n(0) - L(t-t^j_n) \psi^j\|_{L^\infty([0,T],H^1)} + \|v^j_n\|_{L^8([0,T], L^4)}^2 \|v^j_n\|_{L^{8/3}([0,T], W^{1,4})} \\
		&\lesssim \|v^j_n(0) - L(-t^j_n)\psi^j\|_{H^1} + \|v^j_n\|_{L^8([0,T], L^4)}^2 \|v^j_n\|_{L^{8/3}([0,T], W^{1,4})}.
		\end{align*}
		Since  the last factor in the right hand side is bounded by the latter argument,  \eqref{est-v-jn} and \eqref{obser-vj-tn} imply
		\[
		\|v^j_n(t)- L(t-t^j_n) \psi^j\|_{L^\infty([0,T],H^1)} \rightarrow 0
		\]
		as $n\rightarrow \infty$, hence \eqref{est-vj-Lr} holds. 
		
		Next, we define
		\[
		B:= \max \left\{ 1, \lim_{n\rightarrow \infty} \sup_{t\in [0,T]} H(u_n(t)) \right\} <\infty.
		\]
		For each $1\leq j \leq J_2$, we denote $T^j$ the maximal forward time such that $\sup_{t\in [0,T^j]}H(v^j(t)) \leq 2B$. Define
		\[
		\tilde{T}:= 
		\left\{
		\begin{array}{ccl}
		\min_{1\leq j \leq J_2} T^j &\text{if}& J_2 \geq 1, \\
		T &\text{if}& J_2 =0.
		\end{array}
		\right.
		\]
		In what follows, we will show that for all $t\in [0,\tilde{T}]$,
		\begin{align} \label{expan-H-un-tilde}
		H(u_n(t)) = \sum_{j=1}^J H(v^j_n(t))+ H(\tilde{W}^J_n(t)) + o_{J,n}(1),
		\end{align}
where $o_{J,n}(1)$ satisfies $\lim_{J\to\infty} \limsup_{n\to\infty} o_{J,n}(1)=0,$  
uniformly on $0\leq t\leq T$. 
		This implies \eqref{expan-H-un-t} as $T\leq\tilde{T}$. Indeed, if $J_2=0$, then $T=\tilde{T}$ by  definition. Otherwise, if $J_2\geq 1$ and $\tilde{T} < T$, then by \eqref{expan-H-un-tilde} and the definition of $\tilde{T}$, there exists $1\leq j \leq J_2$ such that $T^j=\tilde{T}$ and
		\[
		\sup_{t\in[0,T^j]} H(v^j(t)) = \sup_{t\in [0,T^j]} H(v^j_n(t)) \leq \sup_{t\in[ 0,T^j]} H(u_n(t)) \leq \sup_{t\in [0,T]} H(u_n(t)) \leq B.
		\]
		Here we recall that $t^j_n\equiv 0$ for $1\leq j \leq J_2$. By continuity, it contradicts the maximality of $T^j$. 
		
		To show \eqref{expan-H-un-tilde}, we observe that for $1\leq j\leq J_2$, we have
		\begin{align*}
		\|v^j_n\|_{L^8([0,\tilde{T}],L^4)}=\|v^j\|_{L^8([0,\tilde{T}],L^4)} &\leq \tilde{T}^{1/8} \|v^j\|_{L^\infty([0,\tilde{T}],L^4)} \\
		&\leq \tilde{T}^{1/8} \|v^j\|^{1/4}_{L^\infty([0,\tilde{T}],L^2)} \|\nabla v^j\|^{3/4}_{L^\infty([0,\tilde{T}],L^2)}.
		\end{align*}
		It follows that
		\[
		\|v^j_n\|_{L^8([0,\tilde{T}],L^4)} \lesssim \tilde{T}^{1/8} \|v^j\|^{1/4}_{L^\infty([0,\tilde{T}],L^2)} B^{3/4}.
		\]
		On the other hand, by the conservation of mass and the choice of $v^j$, we have for all $t\in [0,\tilde{T}]$,
		\begin{equation}\label{l2bound}
		\|v^j(t)\|_{L^2} = \|v^j(0)\|_{L^2} = \|\psi^j\|_{L^2} \leq \|u_{0,n}\|_{L^2} \leq 1
		\end{equation}
		which implies
		\begin{align*} 
		\|v^j(t)\|_{L^8([0,\tilde{T}],L^4)} \lesssim \tilde{T}^{1/8} B^{3/4}.
		\end{align*}
		Similarly, by Sobolev embedding, we have
		\[
		\|v^j_n\|_{L^\infty([0,\tilde{T}],L^3)} \lesssim B^{1/2}.
		\]
		Thus for $1\leq j \leq J_2$  we get 
		\begin{align} \label{est-vj-tilde}
		\|v^j_n\|_{L^8([0,\tilde{T}], L^4) \cap L^\infty([0,\tilde{T}],L^3)} \leq C(\tilde{T}, B).
		\end{align}
		Next, we define the approximation 
		\[
		\tilde{u}^J_n(t,x):= \sum_{j=1}^J v^j_n(t,x-x^j_n) + \tilde{W}^J_n(t).
		\]
		We see that $\tilde{u}^J_n(0, x) = u_{0,n}(x)$ and
		\[
		i \partial_t \tilde{u}^J_n + \frac{1}{2} \Delta \tilde{u}^J_n - F(\tilde{u}^J_n) = \tilde{e}^J_n,
		\]
		where
		\[
		\tilde{e}^J_n = \sum_{j=1}^J F(v^j_n(\cdot, \cdot-x^j_n)) - F\left(\sum_{j=1}^J v^j_n (\cdot, \cdot-x^j_n) \right) + F\left( \tilde{u}^J_n - \tilde{W}^J_n(\cdot) \right) - F(\tilde{u}^J_n). 
		\]
		\begin{claim} \label{claim}
			The functions $\tilde{u}^J_n$ satisfy 
			\begin{align} \label{prop-u-Jn-1}
			\limsup_{n\rightarrow \infty} \left(\|\tilde{u}^J_n\|_{L^\infty([0,\tilde{T}],H^1)} + \|\tilde{u}^J_n\|_{L^8([0,\tilde{T}], L^4) \cap L^\infty([0,\tilde{T}],L^3)} \right) \lesssim 1
			\end{align}
			uniformly in $J$ and
			\begin{align} \label{prop-u-Jn-2}
			\lim_{J\rightarrow \infty} \lim_{n\rightarrow \infty} \|\tilde{e}^J_n\|_{L^{8/5}([0,\tilde{T}],W^{1,4/3})} + \|\tilde{e}^J_n\|_{L^{8/3}([0,\tilde{T}], L^{4/3})} =0.
			\end{align}
		\end{claim}
	
		With Claim \ref{claim} at hand, Lemma \ref{lem-stability} implies
		\[ 
		\lim_{J\rightarrow \infty} \lim_{n\rightarrow \infty} \|u_n-\tilde{u}^J_n\|_{L^8([0,\tilde{T}], L^4) \cap L^\infty([0,\tilde{T}],L^3)} =0.
		\] 
		By the H\"older's inequality, Sobolev embedding, \eqref{est-H-un-T}, and \eqref{prop-u-Jn-1}, we see that 
		\begin{align*}
		\|u_n-\tilde{u}^J_n\|_{L^\infty([0,\tilde{T}],L^4)} &\leq \|u_n-\tilde{u}^J_n\|_{L^\infty([0,\tilde{T}],L^3)}^{1/2} \|u_n-\tilde{u}^J_n\|_{L^\infty([0,\tilde{T}],L^6)}^{1/2} \\
		&\lesssim \|u_n-\tilde{u}^J_n\|_{L^\infty([0,\tilde{T}],L^3)}^{1/2}\left(\|u_n\|_{L^\infty([0,\tilde{T}], H^1)} +\|\tilde{u}^J_n\|_{L^\infty([0,\tilde{T}],H^1)}\right)^{1/2} \rightarrow 0
		\end{align*}
		as $J, n \rightarrow \infty$. Moreover, by the same argument as in \cite[Proposition 4.3]{BF} using \eqref{time-space-diver} and \eqref{est-vj-Lr}, we see that for all $t\in [0,\tilde{T}]$,
		\begin{align*}
		-N(\tilde{u}^J_n(t)) &= -\sum_{j=1}^J N(v^j_n(t,x-x^j_n)) - N(\tilde{W}^J_n(t))+ o_{J,n}(1) \\
		&= -\sum_{j=1}^J N(v^j_n(t)) - N(\tilde{W}^J_n(t))  + o_{J,n}(1).
		\end{align*}
		On the other hand, by the conservation of energy and the choice of $v^j$, we have
		\begin{align*}
		E(u_n(t)) = E(u_{0,n}) &= \sum_{j=1}^J E(v^j_n(0)) + E(\tilde{W}^J_n) + o_n(1) \\
		&= \sum_{j=1}^J E(v^j_n(t)) + E(\tilde{W}^J_n(t)) + o_n(1).
		\end{align*}
		Combining the above estimates, we prove \eqref{expan-H-un-tilde}. The proof of Lemma \ref{lem-pytha-along-inls} is complete.	
	\end{proof}

	\begin{proof}[Proof of Claim \ref{claim}]
		We first show the smallness of remainder under the time evolution of \eqref{dip-NLS}. By the Duhamel formula, Strichartz estimates, and Lemma \ref{lem-non-est}, we have 
		\begin{align} \label{est-tilde-W-J-n}
		\|\tilde{W}^J_n(t)\|_{L^8(\R,L^4)\cap L^4(\R,L^6)} \leq \|L(t) \tilde{W}^J_n\|_{L^8(\R,L^4)\cap L^4(\R,L^6)} + C \|\tilde{W}^J_n(t)\|^3_{L^8(\R,L^4)\cap L^4(\R,L^6)}
		\end{align}
		for some constant $C>0$ independent of $n$. By the standard continuity argument  together with  \eqref{remainder-W}, we get
		\begin{align} \label{lim-J-tilde-W-J-n}
		\lim_{J\rightarrow \infty} \left[\lim_{n\rightarrow \infty} \|\tilde{W}^J_n(t)\|_{L^8(\R,L^4)\cap L^4(\R,L^6)} \right] =0.
		\end{align}
				
		The boundedness of $\|\tilde{u}^J_n\|_{L^8([0,\tilde{T}],L^4) \cap L^\infty([0,\tilde{T}],L^3)}$ follows directly from \eqref{obser-vj-tn}, \eqref{est-vj-Lr}, \eqref{est-vj-tilde}, and \eqref{lim-J-tilde-W-J-n}. To see the boundedness of $\|\tilde{u}^J_n\|_{L^\infty([0,\tilde{T}], H^1)}$ uniformly in $J$, we proceed as follows. For $J_2+1 \leq j \leq J_0$, we use the Duhamel formula, Strichartz estimates, \eqref{est-vj-Lr}, and the fact that
		\begin{align*}
		\|\nabla F(u)\|_{L^2([0,\tilde{T}], L^{\frac{6}{5}})} &\lesssim \|u\|^2_{L^4([0,\tilde{T}],L^6)} \|\nabla u\|_{L^\infty([0,\tilde{T}],L^2)} \\
		&\lesssim \tilde{T}^{1/2} \|u\|^2_{L^\infty([0,\tilde{T}],L^6)} \|\nabla u\|_{L^\infty([0,\tilde{T}],L^2)}
		\end{align*}
		to have
		\[
		\|\nabla v^j_n\|_{L^\infty([0,\tilde{T}],L^2)} \lesssim \|\nabla v^j_n(0)\|_{L^2}
		\]
		for $n$ sufficiently large. A similar estimate holds for $j\geq J_0+1$ since $\|v^j_n\|_{L^4(\R, L^6)}$ is small by taking $J_0$ sufficiently large. Similarly, by \eqref{lim-J-tilde-W-J-n}, we have
		\[
		\|\nabla \tilde{W}^J_n(t)\|_{L^\infty([0,\tilde{T}],L^2)} \lesssim \|\nabla \tilde{W}^J_n\|_{L^2}
		\]
		for $J, n$ sufficiently large. Thus, we get
		\begin{align*}
		\|\nabla \tilde{u}^J_n\|^2_{L^\infty([0,\tilde{T}],L^2)} &\leq \sum_{j=1}^{J_2} \|\nabla v^j\|^2_{L^\infty([0,\tilde{T}],L^2)} + \sum_{j=J_2 +1}^J \|\nabla v^j_n\|^2_{L^\infty([0,\tilde{T}],L^2)} \\
		&\mathrel{\phantom{\leq \sum_{j=1}^{J_2} \|\nabla v^j\|^2_{L^\infty([0,\tilde{T}],L^2)}}}+ \|\nabla \tilde{W}^J_n(t)\|^2_{L^\infty([0,\tilde{T}],L^2)} \\
		&\lesssim J_2 B^2 + \sum_{j=J_2+1}^J \|\nabla v^j_n(0)\|^2_{L^2} + \|\nabla \tilde{W}^J_n\|^2_{L^2} +o_{J,n}(1)\\
		&\lesssim J_2 B^2 + \sum_{j=J_2 +1}^J \|\nabla \psi^j\|^2_{L^2} + \|\nabla W^J_n\|^2_{L^2}+ o_{J,n}(1) \\
		&\lesssim J_2 B^2 + \|\nabla u_{0,n}\|^2_{L^2} + o_{J,n}(1) \\
		&\lesssim J_2 B^2 + B^2 +o_{J,n}(1).
		\end{align*}
		Note that $J_2 \leq J_0$ which is independent of $J$ for $J$ large.	On the other hand, by the conservation of mass, we have similarly to \eqref{l2bound}
		\[
		\|v^j_n\|_{L^\infty([0,\tilde{T}],L^2)} = \|v^j_n(0)\|_{L^2} =\lim_{n\rightarrow \infty} \|e^{-it^j_n} \psi^j\|_{L^2} = \|\psi^j\|_{L^2} \leq \|u_{0,n}\|_{L^2} \leq 1.
		\]
		Collecting the above estimates, we show the boundedness of $\|\tilde{u}^J_n\|_{L^\infty([0,\tilde{T}],H^1)}$, hence \eqref{prop-u-Jn-1} follows. 
		
		To see \eqref{prop-u-Jn-2}, it suffices to show
		\begin{align} \label{est-A}
		\lim_{J\rightarrow \infty} \lim_{n\rightarrow \infty} \|A(J,n)\|_{L^{8/5}([0,\tilde{T}],W^{1,4/3})} + \|A(J,n)\|_{L^{8/3}([0,\tilde{T}],L^{4/3})} = 0 
		\end{align}
		and
		\begin{align} \label{est-B}
		\lim_{J\rightarrow \infty} \lim_{n\rightarrow \infty} \|B(J,n)\|_{L^{8/5}([0,\tilde{T}],W^{1,4/3})} + \|B(J,n)\|_{L^{8/3}([0,\tilde{T}],L^{4/3})} = 0,
		\end{align}
		where
		\begin{align*}
		A(J,n) &:= \sum_{j=1}^J F(v^j_n(\cdot, \cdot-x^j_n)) - F\left(\sum_{j=1}^J v^j_n (\cdot, \cdot-x^j_n) \right), \\
		B(J,n) &:= F\left( \tilde{u}^J_n - \tilde{W}^J_n(\cdot) \right) - F(\tilde{u}^J_n).
		\end{align*}
		The estimate \eqref{est-A} follows from standard argument (see e.g. \cite[Proposition 5.4]{HR}). A similar argument goes for \eqref{est-B} using \eqref{lim-J-tilde-W-J-n}. We omit the details.
		\end{proof}

At this point we can distinguish two possible cases.	\\
	
	\noindent \textit{Case 1. More than one non-zero profiles.} We  suppose that there exist at least two nontrivial profiles $\psi^j$. Then we infer from \eqref{H-gam-expan} and \eqref{est-v-jn} that
	\begin{align} \label{M-vj}
	M(v^j_n(t)) = M(v^j_n(0))= M(\psi^j)<1, \quad \forall j\geq 1.
	\end{align} 
	
	\noindent Moreover, by \eqref{prop-un-0}, \eqref{expan-N-un-t}, and \eqref{M-vj}, we have
	\begin{align} \label{key}
	\sup_{t\in [0,\infty)} -N(v^j_n(t)) M(v^j_n(t))< A, \quad E(v^j_n(t))M(v^j_n(t)) <\deltc
	\end{align}
	which, by Item $\mathcal A)$ (see page \pageref{case-A:page}), implies that 
	\[
	\|v^j_n\|_{L^8([0,\infty), L^4)} <\infty, \quad \forall j \geq 1.
	\]
	We can approximate
	\[
	u^J_n(t,x) \sim \sum_{j=1}^J v^j_n(t, x-x^j_n)
	\]
	using the long time perturbation argument and get for $J$ sufficiently large,
	\[
	\|u_n\|_{L^8([0,\infty), L^4)} <\infty
	\]
	which is a contradiction. \\
	
	\noindent \textit{Case 2. Only one non-zero profile.} Therefore we now have only one non-zero profile, namely 
	\[
	u_{0,n}(x) = e^{-it^1_n\Delta} \psi^1(x-x^1_n) + W_n(x), \quad \lim_{n\rightarrow \infty} \|e^{it\Delta} W_n\|_{L^8(\R, L^4)} =0.
	\]
	We claim that we cannot have $t^1_n \rightarrow -\infty$.  Indeed, suppose that $t^1_n \rightarrow -\infty$. It follows that	
	\[
	\|e^{it\Delta} u_{0,n}\|_{L^8([0,\infty),L^4)} \leq \|e^{it\Delta} \psi^1\|_{L^8([-t^1_n, \infty),L^4)} + \|e^{it\Delta} W_n\|_{L^8([0,\infty),L^4)} \rightarrow 0 \text{ as } n\rightarrow \infty.
	\]
	
\noindent The continuity argument yields $\|u_n\|_{L^8([0,\infty),L^4)} < \infty$ for $n$ sufficiently large. By Lemma \ref{lem-scat-cond}, $u_n(t)$ scatter in $H^1$ forward in time, which is a contradiction. 
	
	Let $v^1$ be the corresponding nonlinear profile associated to $\psi^1$ and set $v^1_n(t):= v^1(t-t^1_n)$. Accordingly to \eqref{non-prof-decom-un-0}, we have
	\[
	u_{0,n}(x) = v^1_n(0, x-x^1_n) + \tilde{W}^1_n(x).
	\]
	Arguing as above, we have that $v^1_n$ and $\tilde{W}^1_n(t):=\nls(t)\tilde{W}^1_n$ satisfy
	\[
	M(v^1_n(t)) \leq 1, \quad \sup_{t\in [0,\infty)} -N(v^1_n(t)) \leq A, \quad E(v^1_n(t)) \leq \deltc
	\]
	and
	\[
	\lim_{n\rightarrow \infty} \|\tilde{W}^1_n(t)\|_{L^8(\R,L^4)} =0.
	\]
	We infer that 
	\[
	M(v^1_n(t)) = 1, \quad E(v^1_n(t)) = \deltc.
	\]
	Indeed, if $M(v^1_n(t))<1$, then
	\[
	\sup_{t\in  [0,\infty)} -N(v^1_n(t)) M(v^1_n(t))< A, \quad E(v^1_n(t))M(v^1_n(t)) <\deltc.
	\]	
\noindent By Item $\mathcal A),$ (see page \pageref{case-A:page}), we have
	\[
	\|v^1_n\|_{L^8([0,\infty),L^4)} <\infty
	\]
	which, by the long time perturbation argument, implies
	\[
	\|u_n\|_{L^8([0,\infty),L^4)} <\infty.
	\]
	We get a contradiction, and similarly we can discard the possibility that $E(v^1_n(t)) < \deltc.$ We now define $u_{\csb}$ the solution to \eqref{dip-NLS} with initial data $v^1(0)$. We see that $u_{\csb}$ satisfies \eqref{prop-uc}. Indeed, we have
	\begin{align*}
	M(u_{\csb}) &= M(v^1(0)) = M(v^1(t)) = M(v^1_n(t)) = 1, \\
	E(u_{\csb}) &= E(v^1(0)) = E(v^1(t)) = E(v^1_n(t)) = \deltc
	\end{align*}
	and
	\[
	\sup_{t\in[0,\infty)} -N(u_{\csb}(t)) = \sup_{t\in [0,\infty)} -N(v^1(t)) =\sup_{t \in [t^1_n,\infty)} -N(v^1_n(t)) \leq A.
	\]
	By the definition of $\deltc$, we must have $\|u_{\csb}\|_{L^8([0,\infty),L^4)} = \infty$. The construction of a critical (i.e. global and non-scattering) solution (see \eqref{prop-uc}) is done. \\
	
	\noindent \textit{Step 3. Exclusion of the critical solution.} By the compactness argument similar to \cite{BF}, we show that there exists a continuous path $x(t) \in \R^3$ which grows sub-linearly in time, such that
	\[
	\mathcal{O}:= \left\{ u_{\csb}(t,\cdot-x(t)) \ : \ t \in [0,\infty)\right\} 
	\]
	is precompact in $H^1$. Using this compactness result, the rigidity argument using localized virial estimates and Lemma $\ref{lem-bound-below}$ shows that $u_{\csb}(t) \equiv 0$ which contradicts \eqref{prop-uc}.\\
	
	This in turn also concludes the proof of Proposition \ref{prop-scat-equi}, hence Theorem \ref{theo-scat-crite} follows.

\begin{remark} \label{rem-cond-GW}
		The condition \eqref{cond-GW} is needed in the proof of Theorem \ref{theo-scat-crite}. It ensures that the functional $N$ is always negative which 
		is crucial to get \eqref{key}. If there is a profile $v^{j_0}_n$ such that $N(v^{j_0}_n(t_0))$ becomes non-negative at time $t_0$, then there may be another $v^{j_1}_n$ satisfying $N(v^{j_1}_n(t_0)) \geq A$. In this case, we cannot conclude.
	\end{remark}
	
	\subsection{Scattering above the threshold}
	We next show the energy scattering result given in Theorem \ref{theo-scat-above}.  as the first consequence of the more general Theorem \ref{theo-scat-crite}. \\
	
	\noindent {\it Proof of Theorem \ref{theo-scat-above}.} Let $u_0 \in \Sigma$ satisfying \eqref{cond-1-above}, \eqref{cond-2-above}, \eqref{cond-scat-1-above} and \eqref{cond-scat-2-above}. We will show that \eqref{scat-crite} holds true, which in turn  implies the result, by means of Theorem \ref{theo-scat-crite}. It is done in several steps. The strategy is  in the spirit of Duyckaerts and Roudenko \cite{DR-beyond}. \\
	
	\noindent \textit{Step 1. Reduction of conditions.} We first recall the following estimate due to Gao-Wang \cite{GW}:
	\begin{align} \label{est-GW}
	\left(\ima \int_{\R^3} \overline{f}(x) x\cdot \nabla f(x) dx \right)^2 \leq \|x f\|^2_{L^2} \left( H(f) - \frac{(-N(f))^{\frac{2}{3}}}{(C_{\opt})^{\frac{2}{3}} (M(f))^{\frac{1}{3}}} \right), \quad f\in \Sigma,
	\end{align}
	where $C_{\opt}$ is the sharp constant in \eqref{GN-ineq}. Let $V(t)$ be as in \eqref{defi-V}. It is known (see e.g. \cite{CMS}) that
	\begin{align} \label{deri-V}
	V'(t) = 2\ima \int_{\R^3} \overline{u}(t,x) x \cdot \nabla u(t,x) dx, \quad V''(t) = 2 H(u(t)) + 3N(u(t)).
	\end{align}
	In particular, we have
	\begin{align} \label{deri-seco-V}
	V''(t) = 4E(u) + N(u(t)) = 6E(u) - H(u(t)).
	\end{align}
	It follows that
	\begin{align} \label{iden-N-H}
	-N(u(t)) = 4E(u) - V''(t), \quad H(u(t)) = 6E(u) - V''(t).
	\end{align}
	Thanks to \eqref{cond-1-GW}, $N(u(t))$ takes negative values, hence $V''(t) \leq 4E(u)$ for all times $t$ on the lifespan of the solution. Inserting \eqref{iden-N-H} into \eqref{est-GW}, we infer that
	\begin{align*}
	\left(\frac{V'(t)}{2}\right)^2 \leq V(t) \left[ 6E(u)-V''(t) - \frac{(4E(u)-V''(t))^{\frac{2}{3}}}{(C_{\opt})^{\frac{2}{3}} (M(u))^{\frac{1}{3}}} \right]
	\end{align*}
	which implies
	\begin{align} \label{est-z}
	(z'(t))^2 \leq h(V''(t)),
	\end{align}
	where $z(t):= \sqrt{V(t)}$	and
	\[
	h(\lambda):= 6E(u)- \lambda - \frac{(4E(u)-\lambda)^{\frac{2}{3}}}{(C_{\opt})^{\frac{2}{3}} (M(u))^{\frac{1}{3}}}
	\]
	with $\lambda \leq 4E(u)$. We see that the real function $h$ is decreasing on $(-\infty, \lambda_0)$ and increasing on $(\lambda_0, 4E(u))$, where $\lambda_0$ satisfies
	\begin{align} \label{defi-lambda-0}
	1= \frac{2(4E(u)-\lambda_0)^{-\frac{1}{3}}}{3 (C_{\opt})^{\frac{2}{3}} (M(u))^{\frac{1}{3}}}.
	\end{align}
	This implies that
	\[
	h(\lambda_0) = \frac{\lambda_0}{2}.
	\]
	Using \eqref{opti-cons} and \eqref{inde-quant-proof}, we infer from \eqref{defi-lambda-0} that
	\begin{align} \label{iden-lambda-0}
	\frac{E(u) M(u)}{E(\phi) M(\phi)} \left(1-\frac{\lambda_0}{4E(u)}\right) =1.
	\end{align}
	As consequence of the above identity and the conservation laws of mass and energy, we see that the assumption \eqref{cond-1-above} is equivalent to 
	\begin{align} \label{cond-1-above-equi}
	\lambda_0 \geq 0
	\end{align}
	and the assumption \eqref{cond-2-above} is equivalent to $(V'(0))^2 \geq 2 \lambda_0 V(0)$ or
	\begin{align} \label{cond-2-above-equi}
	(z'(0))^2 \geq \frac{\lambda_0}{2} = h(\lambda_0).
	\end{align}
	Moreover, the assumption \eqref{cond-scat-1-above} is equivalent to 
	\begin{align} \label{cond-scat-1-above-equi}
	V''(0) >\lambda_0.
	\end{align}
	Indeed, by \eqref{cond-scat-1-above}, \eqref{inde-quant-proof} and \eqref{iden-lambda-0}, we have
	\begin{align*}
	V''(0) &= 4E(u_0) + N(u_0) \\
	&= 4E(u_0) + \frac{N(u_0) M(u_0)}{M(u_0)} \\
	&> 4E(u_0) + \frac{N(\phi) M(\phi)}{M(u_0)} \\
	& = 4E(u_0) \left(1-\frac{-N(\phi) M(\phi)}{4E(u_0) M(u_0)} \right)\\
	& = 4E(u_0) \left( 1- \frac{E(\phi) M(\phi)}{E(u_0) M(u_0)} \right) \\
	&= \lambda_0.
	\end{align*}
	In addition, the assumption \eqref{cond-scat-2-above} is equivalent to 
	\begin{align} \label{cond-scat-2-above-equi}
	z'(0) \geq 0.
	\end{align}
	
	\noindent \textit{Step 2. Lower bound of $V''(t)$.}
	We claim that there exists $\delta_0>0,$ possibly small, such that for all $t\in [0,T^*)$,
	\begin{align} \label{claim-above}
	V''(t) \geq \lambda_0 + \delta_0.
	\end{align}
	By \eqref{cond-scat-1-above-equi}, we take $\delta_1>0$ so that 
	\[
	V''(0) \geq \lambda_0 + 2\delta_1.
	\]
	By continuity, we have
	\begin{align} \label{claim-above-proof-1}
	V''(t) >\lambda_0 + \delta_1, \quad \forall t\in [0,t_0).
	\end{align}
	for $t_0>0$ sufficiently small. By reducing $t_0$ if necessary, we can assume that
	\begin{align} \label{claim-above-proof-2}
	z'(t_0) >\sqrt{h(\lambda_0)}.
	\end{align}
	Indeed, if $z'(0) >\sqrt{h(\lambda_0)}$, then \eqref{claim-above-proof-2} follows from the continuity argument. Otherwise, if $z'(0)=\sqrt{h(\lambda_0)}$, then using the identity
	\begin{align} \label{claim-above-proof-3}
	z''(t) = \frac{1}{z(t)} \left(\frac{V''(t)}{2} - (z'(t))^2\right)
	\end{align}
	and \eqref{cond-scat-1-above-equi}, we have $z''(0)>0$. This shows \eqref{claim-above-proof-2} by taking $t_0>0$ sufficiently small. Thanks to \eqref{claim-above-proof-2}, we take $\epsilon_0>0$  small enough so that
	\begin{align} \label{claim-above-proof-4}
	z'(t_0) \geq \sqrt{h(\lambda_0)} + 2\epsilon_0.
	\end{align}
	We will prove by contradiction that 
	\begin{align} \label{claim-above-proof-5}
	z'(t) > \sqrt{h(\lambda_0)} + \epsilon_0, \quad \forall t\geq t_0.
	\end{align}
	Suppose that it is not true and set
	\[
	t_1:= \inf \left\{ t\geq t_0 \ : \ z'(t) \leq \sqrt{h(\lambda_0)} + \epsilon_0 \right\}.
	\]
	By \eqref{claim-above-proof-4}, we have $t_1 >t_0$. By continuity, we have
	\begin{align} \label{claim-above-proof-6}
	z'(t_1) = \sqrt{h(\lambda_0)} + \epsilon_0
	\end{align}
	and
	\begin{align} \label{claim-above-proof-7}
	z'(t) \geq \sqrt{h(\lambda_0)} + \epsilon_0, \quad \forall t\in [t_0,t_1].
	\end{align}
	By \eqref{est-z}, we see that
	\begin{align} \label{claim-above-proof-8}
	\left(\sqrt{h(\lambda_0)} + \epsilon_0\right)^2 \leq (z'(t))^2 \leq h(V''(t)), \quad \forall t\in [t_0,t_1].
	\end{align}
	It follows that $h(V''(t)) > h(\lambda_0)$ for all $t\in [t_0,t_1]$, thus $V''(t) \ne \lambda_0$ and by continuity, $V''(t) >\lambda_0$ for all $t\in [t_0,t_1]$. 
	
	We will prove that there exists a constant $C>0$ such that
	\begin{align} \label{claim-above-proof-9}
	V''(t) \geq \lambda_0 +\frac{\sqrt{\epsilon_0}}{C}, \quad \forall t\in [t_0,t_1].
	\end{align}
	Indeed, by the Taylor expansion of $h$ near $\lambda_0$ with the fact $h'(\lambda_0)=0$, there exists $a>0$ such that
	\begin{align} \label{claim-above-proof-10}
	h(\lambda) \leq h(\lambda_0) + a(\lambda-\lambda_0)^2, \quad \forall \lambda \hbox{\quad s.t. \quad}  |\lambda-\lambda_0| \leq 1.
	\end{align}
	If $V''(t) \geq \lambda_0+1$, then \eqref{claim-above-proof-9} holds by taking $C$ large. If $\lambda_0<V''(t) \leq \lambda_0+1$, then by \eqref{claim-above-proof-8} and \eqref{claim-above-proof-10}, we get
	\[
	\left(\sqrt{h(\lambda_0)} +\epsilon_0\right)^2 \leq (z'(t))^2 \leq h(V''(t)) \leq h(\lambda_0) + a (V''(t)-\lambda_0)^2
	\]
	thus
	\[
	2\epsilon_0\sqrt{h(\lambda_0)} + \epsilon_0^2 \leq a (V''(t)-\lambda_0)^2.
	\]
	This shows \eqref{claim-above-proof-9} with $C=\sqrt{\frac{a}{2}} [h(\lambda_0)]^{-\frac{1}{4}}$. 
	
\noindent	However, by \eqref{claim-above-proof-3}, \eqref{claim-above-proof-6} and \eqref{claim-above-proof-9}, we have
	\begin{align*}
	z''(t_1) &= \frac{1}{z(t_1)} \left(\frac{V''(t_1)}{2} - (z'(t_1))^2\right) \\
	&\geq \frac{1}{z(t_1)} \left(\frac{\lambda_0}{2} + \frac{\sqrt{\epsilon_0}}{2C}- \left(\sqrt{h(\lambda_0)} +\epsilon_0\right)^2 \right) \\
	&\geq \frac{1}{z(t_1)} \left(\frac{\sqrt{\epsilon_0}}{2C} -2\epsilon_0 \sqrt{h(\lambda_0)} - \epsilon_0^2\right)>0
	\end{align*}
	provided $\epsilon_0$ is taken small enough, and this contradicts \eqref{claim-above-proof-6} and \eqref{claim-above-proof-7}. This proves \eqref{claim-above-proof-5}. Note that we have also proved \eqref{claim-above-proof-9} for all $t\in [t_0,T^*)$. This together with \eqref{claim-above-proof-1} imply \eqref{claim-above} with $\delta_0=\min \left\{\delta_1,\frac{\sqrt{\epsilon_0}}{C}\right\}$.\\
	
	\noindent \textit{Step 3. Conclusion.}
	Eventually, we are able to prove \eqref{scat-crite}. It follows from \eqref{claim-above} that
	\begin{align*}
	-N(u(t)) M(u(t)) &= (4 E(u) - V''(t)) M(u) \\
	&\leq (4E(u) -\lambda_0 -\delta_0) M(u) \\
	&\leq 4 E(\phi) M(\phi) - \delta_0 M(u) \\
	& = -(1-\eta) N(\phi) M(\phi) 
	\end{align*}
	for all $t\in [0,T^*)$, where $\eta:= \frac{\delta_0M(u)}{4E(\phi) M(\phi)}>0$. Here we have used \eqref{iden-lambda-0} to get the third line and \eqref{inde-quant-proof} to get the last line. This shows \eqref{scat-crite} and the proof of Theorem \ref{theo-scat-above} is complete.
	
\subsection{Construction on initial data as in Theorem \ref{theo-scat-above}}	
We conclude this Section by showing  the existence of initial data  satisfying the conditions of Theorem \ref{theo-scat-above}. For the reader's convenience,  we recall that we are looking to functions  $u_0 \in \Sigma$ such that
\begin{align}
E(u_0) M(u_0) &\geq E(\phi) M(\phi), \label{cond-1} \\
\frac{E(u_0) M(u_0)}{E(\phi) M(\phi)} &\left(1-\frac{(V^\prime(0))^2}{8E(u_0) V(0)}\right) \leq 1, \label{cond-2} \\
- N(u_0) M(u_0) &< - N(\phi) M(\phi), \label{cond-3} \\
V'(0) &\geq 0. \label{cond-4}
\end{align}
To this end, we follow an idea of Duyckaerts and Roudenko \cite{DR-beyond}. Let $v_0 \in\Sigma$ and denote
\begin{equation}\label{in-datum-osc}
u_0(x):= e^{i \mu |x|^2} v_0(x), \quad \mu \in \R
\end{equation}
A direct computation shows
\begin{align*}
M(u_0) &= \|u_0\|^2_{L^2} = \|v_0\|^2_{L^2} = M(v_0), \\
N(u_0) &= \int \lambda_1 |u_0(x)|^4 +\lambda_2 (K\ast |u_0(x)|^2) |u_0(x)|^2 dx \\
&= \int \lambda_1 |v_0(x)|^4 +\lambda_2 (K\ast |v_0(x)|^2) |v_0(x)|^2 dx = N(v_0), 
\end{align*}
and
\[
H(u_0) = \|\nabla u_0\|^2_{L^2} = \|\nabla v_0\|^2_{L^2} + 4 \mu^2 \|xv_0\|^2_{L^2} + 4 \mu \ima \int x \cdot \nabla v_0(x) \overline{v}_0(x) dx.
\]
We also have $\|xu_0\|^2_{L^2} = \|xv_0\|^2_{L^2}$, and, by denoting with $u(t)$ the solution to \eqref{dip-NLS} with initial datum $u_0$ as in \eqref{in-datum-osc}, we have 
\[
V^\prime(0) = 2 \ima \int x \cdot \nabla u_0(x) \overline{u}_0(x) dx = 2 \ima \int x \cdot \nabla v_0(x) \overline{v}_0(x) dx + 4 \mu \|xv_0\|^2_{L^2}.
\]
Observe that if $v_0$ is real-valued, then 
\begin{align*}
E(u_0) = E(v_0) + 2\mu^2 \|xv_0\|^2_{L^2}, \quad V^\prime(0) = 4 \mu \|xv_0\|^2_{L^2} .
\end{align*}
The conditions \eqref{cond-1}--\eqref{cond-4} now become
\begin{align}
\left(E(v_0) + 2\mu^2 \|xv_0\|^2_{L^2}\right) M(v_0) &\geq E(\phi) M(\phi), \label{cond-1-equi} \\
E(v_0) M(v_0) &\leq E(\phi) M(\phi), \label{cond-2-equi} \\
- N(v_0) M(v_0) &< - N(\phi) M(\phi), \label{cond-3-equi} \\
\mu \|xv_0\|^2_{L^2} &\geq 0.\label{cond-4-equi}
\end{align}
Note that condition \eqref{cond-4-equi} is satisfied for all $\mu>0$. Let us take
\[
v_0(x) = \lambda^{\frac{5}{2}} \phi(\lambda x), \quad \lambda>0,
\]
where $\phi$ is a non-negative ground state related to \eqref{ell-equ}. Note that such a ground state exists due to \cite{AS}. A calculation shows
\[
M(v_0) = \lambda^2 M(\phi), \quad N(v_0) = \lambda^7 N(\phi), \quad H(v_0) = \lambda^4 H(\phi), \quad \|xv_0\|^2_{L^2} = \|x\phi\|^2_{L^2}.
\]
By plugging the above identities into \eqref{cond-1-equi}--\eqref{cond-4-equi}, we get
\begin{align}
\left(\frac{\lambda^4}{2} H(\phi) + \frac{\lambda^7}{2}N(\phi) + 2\mu^2 \|x\phi\|^2_{L^2}\right) \lambda^2 M(\phi) &\geq E(\phi)M(\phi), \label{cond-lamb-1} \\
\left(\frac{\lambda^4}{2} H(\phi) + \frac{\lambda^7}{2}N(\phi)\right) \lambda^2 M(\phi) &\leq E(\phi) M(\phi), \label{cond-lamb-2} \\
- \lambda^9 N(\phi) M(\phi) &<-N(\phi) M(\phi), \label{cond-lamb-3} \\
\mu\|x\phi\|^2_{L^2} &\geq 0. \label{cond-lamb-4} 
\end{align}
Since $-N(\phi)>0$ and $E(\phi)>0$, we see that \eqref{cond-lamb-2} and \eqref{cond-lamb-3} are satisfied for $\lambda>0$ sufficiently small. Once $\lambda$ is fixed, we take $\mu>0$ sufficiently large so that \eqref{cond-lamb-1} and \eqref{cond-lamb-4} are fulfilled. This shows the existence of initial data satisfying \eqref{cond-1}--\eqref{cond-4}. The proof is complete.

	\section{Blow-up or Grow-up result}\label{S4}
	\setcounter{equation}{0}
	
	The blow-up criterion given in Theorem \ref{theo-blow-crite} follows from the  Du-Wu-Zhang argument, see \cite{DWZ}, and some control on the non-local dipolar term proved by the second author and Bellazzini in \cite{BF}. 
	
	Let us start by recalling the following localized virial identities related to \eqref{dip-NLS}. Let $\chi$ be a smooth radial function satisfying
	\begin{align} \label{defi-chi}
	\chi(x) =\chi(r)= \left\{
	\begin{array}{ccc}
	r^2 &\text{if}&  r \leq 1, \\
	0 &\text{if} & r \geq 2,
	\end{array}
	\right.
	\quad \chi''(r) \leq 2 \text{ for all } r=|x| \geq 0.
	\end{align}
	Given $R>1$, we define the radial function
	\begin{align} \label{defi-varphi-R}
	\varphi_R(x):= R^2 \chi(x/R).
	\end{align} 
	We define the following localized virial quantity
	\begin{align} \label{defi-z-R}
	z_{\varphi_R}(t):= \int_{\R^3} \varphi_R(x) |u(t,x)|^2 dx.
	\end{align}
	We have the following localized virial estimate.
	\begin{proposition} \label{prop-loca-viri-est}
		It holds that
		\[
		z'_{\varphi_R}(t) =  \ima \int_{\R^3} \overline{u}(t) \nabla \varphi_R \cdot \nabla u(t) dx
		\]
		and
		\begin{equation}\label{loca-viri-est}
		z''_{\varphi_R}(t) \leq 2 G(u(t)) + A_R(t),
		\end{equation}
		where
		\begin{equation}
		\begin{aligned} \label{est-A-R}
		|A_R(t)|&\lesssim  R^{-1} +R^{-1}\|u(t)\|^2_{H^1} \|u(t)\|^2_{L^4(|x| \gtrsim R)} + R^{-1} \|u(t)\|^2_{H^1} \\
		&+\|u(t)\|^2_{L^4(|x| \gtrsim R)} +\|u(t)\|^4_{L^4(|x| \gtrsim R)},	\end{aligned}
		\end{equation}
		for all $t\in[0,T^*)$.
	\end{proposition}
	
	\begin{proof} Straightforward computations show that 
		\[
		z'_{\varphi_R}(t) =  \ima \int_{\R^3} \overline{u}(t) \nabla \varphi_R \cdot \nabla u(t) dx 
		\]
		and
		\begin{align*}
		z''_{\varphi_R}(t) &= -\frac{1}{4}\int_{\R^3} \Delta^2 \varphi_R |u(t)|^2 dx +  \sum_{j,k=1}^3 \rea \int_{\R^3} \partial^2_{jk} \varphi_R \partial_j \overline{u}(t) \partial_k u(t) dx \\
		&\mathrel{\phantom{=}} + \frac{\lambda_1}{2} \int_{\R^3} \Delta \varphi_R |u(t)|^4 dx - \lambda_2 \int_{\R^3} \nabla \varphi_R \cdot \nabla (K\ast |u(t)|^2) |u(t)|^2 dx.
		\end{align*}
		Using the fact $\|\Delta^2\varphi_R\|_{L^\infty} \lesssim R^{-2}$ and the conservation of mass, we have
		\[
		\int_{\R^3} \Delta^2\varphi_R |u(t)|^2 dx \lesssim R^{-2}.
		\]
		Since $\varphi_R$ is radial, we use the fact
		\[
		\partial_j = \frac{x_j}{r} \partial_r, \quad \partial^2_{jk} = \left(\frac{\delta_{jk}}{r}-\frac{x_jx_k}{r^3}\right) \partial_r + \frac{x_jx_k}{r^2} \partial^2_r
		\]
		to write
		\begin{align*}
		\sum_{j,k=1}^3 \rea \int_{\R^3} \partial^2_{jk} \varphi_R \partial_j \overline{u}(t) &\partial_k u(t) dx \\
		&= \int_{\R^3} \frac{\varphi'_R(r)}{r} |\nabla u(t)|^2 dx + \int_{\R^3} \left(\frac{\varphi''_R(r)}{r^2}-\frac{\varphi'_R(r)}{r^3} \right)|x\cdot \nabla u(t)|^2 dx.  
		\end{align*}
		Since $\varphi''_R(r)\leq 2$, the Cauchy-Schwarz inequality implies
		\[
		\sum_{j,k=1}^3 \rea \int_{\R^3} \partial^2_{jk} \varphi_R \partial_j \overline{u}(t) \partial_k u(t) dx  \leq 2\|\nabla u(t)\|^2_{L^2}. 
		\]
		Since $\varphi_R(x) = |x|^2$ for $|x| \leq R$, we see that
		\begin{align*}
		\int_{\R^3} \Delta \varphi_R |u(t)|^4 dx &= 6 \|u(t)\|^4_{L^4} + \int_{|x|>R} (\Delta \varphi_R -6) |u(t)|^4 dx \\
		&= 6\|u(t)\|^4_{L^4} + O\left( \|u(t)\|^4_{L^4(|x|>R)}\right).
		\end{align*}
		It follows that
		\begin{equation}\label{est-z-R}
		\begin{aligned}
		z''_{\varphi_R}(t) \leq 2 \|\nabla u(t)\|^2_{L^2} + 3\lambda_1 \|u(t)\|^4_{L^4} -\lambda_2 B_R(t) + O(R^{-2}) + O\left( \|u(t)\|^4_{L^4(|x|>R)}\right), 
		\end{aligned}
		\end{equation}
		where
		\[
		B_R(t):= \int_{\R^3} \nabla\varphi_R \cdot \nabla (K\ast |u(t)|^2) |u(t)|^2 dx.
		\]
		
As we  proved in \cite[Section 6]{BF}, we can estimate the non-local term $B_R(t)$ in the following fashion:
\begin{align}\label{est-B-term}
B_R(t)= &- 3\int (K\ast |u(t)|^2) |u(t)|^2 dx \nonumber \\
&+ O \Big(R^{-1}\|u(t)\|_{H^1}^2\|u(t)\|^2_{L^4(|x| \gtrsim R)}+R^{-3}\|u(t)\|_{H^1}^2\|u(t)\|_{L^2(|x|\gtrsim R)}^2 \\
&\left.\quad +R^{-3}\|u(t)\|_{H^1}\|u(t)\|_{L^2(|x| \gtrsim R)}\|u(t)\|_{L^2}^2+R^{-1}\|u(t)\|^2_{L^4(|x| \gtrsim R)} +\|u(t)\|^2_{L^4(|x| \gtrsim R)}\right). \nonumber
\end{align}
Gluing together \eqref{est-z-R} and \eqref{est-B-term}, and using the Sobolev embedding and the conservation of the mass, we get 
		\begin{align*}
		z''_{\varphi_R}(t) &\leq 2 \|\nabla u(t)\|^2_{L^2} + 3\lambda_1 \|u(t)\|^4_{L^4} + 3 \lambda_2 \int (K\ast |u(t)|^2) |u(t)|^2 dx   \\
		&\mathrel{\phantom{\leq}} + O \Big(R^{-1} + R^{-1}\|u(t)\|^2_{H^1} \|u(t)\|^2_{L^4(|x| \gtrsim R)} + R^{-1} \|u(t)\|^2_{H^1} \\
		&\mathrel{\phantom{\leq + O \Big(R^{-1}}} +\|u(t)\|^2_{L^4(|x| \gtrsim R)} + \|u(t)\|^4_{L^4(|x| \gtrsim R)}  \Big)
		\end{align*}
		for all $t\in [0,T^*)$. Here we have used the fact that $R^{-1}\geq R^{-\beta}$, for $\beta\geq1$. The proof is complete.\\
	\end{proof}
	
	We are now able to prove the blow-up or grow-up result given in Theorem \ref{theo-blow-crite}.\\
	
	\noindent {\it Proof of Theorem \ref{theo-blow-crite}.}
	Let $u:[0,T^*)\times \R^3 \rightarrow \C$ be a $H^1$-solution to \eqref{dip-NLS} satisfying \eqref{blow-crite}. If $T^*<\infty$, then we are done. If $T^*=\infty$, we will show that there exists a time sequence $t_n \rightarrow \infty$ such that 
	\[
	\|\nabla u(t_n)\|_{L^2} \rightarrow \infty
	\] 
	as $n\rightarrow \infty$. Suppose that it is not true. Then we must have
	\begin{align} \label{bound-solu}
	\sup_{t\in [0,\infty)} \|\nabla u(t)\|_{L^2} <\infty.
	\end{align}
	Using \eqref{loca-viri-est}, \eqref{est-A-R}, \eqref{bound-solu}, and the Gagliardo-Nirenberg's inequality we see that
	\begin{align} \label{est-z-R-appl}
	z''_{\varphi_R}(t) \leq 2 G(u(t)) + \tilde A_R(t),
	\end{align}
	where
	\begin{align} \label{est-A-R-appl}
	|\tilde A_R(t)|\lesssim R^{-1} + \|u(t)\|^{1/2}_{L^2(|x| \gtrsim R)}+\|u(t)\|_{L^2(|x| \gtrsim R)}.
	\end{align}
	It remains to estimate $\|u(t)\|_{L^2(|x| \gtrsim R)}$. To do this, we employ  the  technique developed in \cite{DWZ} in the following way. Let $\vartheta$ be a smooth radial function satisfying
	\[
	\vartheta (x) = \vartheta(r) = \left\{
	\begin{array}{ccc}
	0 &\text{if}& r \leq \frac{c}{2}, \\
	1 &\text{if}& r \geq c,
	\end{array}
	\right.
	\quad \vartheta'(r) \leq 1 \text{ for all } r=|x| \geq 0,
	\]
	where $c>0$ is a given constant. For $R>1$, we denote the radial function
	\[
	\psi_R(x) = \psi_R(r) := \vartheta(r/R).
	\]
	By the fundamental theorem of calculus, we have
	\[
	z_{\psi_R}(t) = z_{\psi_R}(0) + \int_0^t z'_{\psi_R}(s) ds \leq z_{\psi_R}(0) + t\sup_{s\in [0,t]} |z'_{\psi_R}(s)|.
	\]
	By the Cauchy-Schwarz's inequality, the conservation of mass and \eqref{bound-solu}, we estimate
	\begin{align*}
	\sup_{s\in [0,t]} |z'_{\psi_R}(s)| &= \sup_{s\in[0,t]} \Big|\ima \int \overline{u}(s) \nabla \psi_R \cdot \nabla u(s) dx \Big| \\
	&\leq \frac{\|u_0\|_{L^2}}{R} \sup_{s\in [0,t]} \|\nabla u(s)\|_{L^2} \\
	&\leq CR^{-1}
	\end{align*}
	for some constant $C>0$. Thus, we get
	\[
	z_{\psi_R}(t) \leq z_{\psi_R}(0) + CR^{-1} t.
	\]
	By the choice of $\vartheta$, we have
	\[
	z_{\psi_R}(0) = \int \psi_R(x) |u_0(x)|^2 dx \leq \int_{|x| \geq \frac{cR}{2}} |u_0(x)|^2 dx \rightarrow 0
	\]
	as $R \rightarrow \infty$ or $z_{\psi_R}(0) = o_R(1)$. Using the fact 
	\[
	\int_{|x| \geq cR} |u(t,x)|^2 dx \leq z_{\psi_R}(t),
	\]
	we obtain the following control on $L^2$-norm of the solution outside a large ball: for any $\eta>0$, there exists a constant $C>0$ independent of $R$ such that for any $t\in [0,T]$ with $T:= \frac{\eta R}{C}$,
	\begin{align} \label{est-L2-norm}
	\int_{|x|\gtrsim R} |u(t,x)|^2 dx \leq \eta + o_R(1).
	\end{align}
	Combining \eqref{est-z-R-appl}, \eqref{est-A-R-appl} and \eqref{est-L2-norm}, we obtain that 
	for any $\eta \in (0,1)$, there exists a constant $C>0$ independent of $R$ such that for any $t\in [0,T]$ with $T:= \frac{\eta R}{C}$ such that
	\[
	z''_{\varphi_R}(t) \leq 2 G(u(t)) + O\left(\left(\eta + o_R(1)\right)^{1/4}+\left(\eta + o_R(1)\right)^{1/2}\right)
	\]
	By the assumption \eqref{blow-crite}, we choose $\eta>0$ sufficiently small and $R>1$ sufficiently large to have
	\[
	z''_{\varphi_R}(t) \leq -\delta <0
	\]
	for all $t\in [0,T]$. Integrating twice from 0 to $T$, we get
	\begin{align*}
	z_{\varphi_R}(T) \leq z_{\varphi_R}(0) + z'_{\varphi_R}(0) T -\frac{\delta}{2} T^2 = z_{\varphi_R}(0) + z'_{\varphi_R}(0) \frac{\eta R}{C} - \frac{\delta \eta^2}{2C^2} R^2.
	\end{align*}
	Noting that $z_{\varphi_R}(0)= o_R(1) R^2$ and $z'_{\varphi_R}(0) = o_R(1) R$, we see that
	\[
	z_{\varphi_R}(T) \leq o_R(1) R^2 - \frac{\delta \eta^2}{2C^2} R^2.
	\]
	Taking $R>1$ large enough, we obtain $z_{\varphi_R}(T) \leq -\frac{\delta \eta^2}{4C}R^2 <0,$ thus we get a contradiction. The proof is complete.

	\section{Long time dynamics  at the mass-energy threshold}
	\label{S5}
	\setcounter{equation}{0}
	
	In this section, we study long time dynamics at the mass-energy threshold given in Theorem \ref{theo-dyna-at}. To do this, we need the following compactness of minimizing sequence for the Gagliardo-Nirenberg's inequality \eqref{GN-ineq}.

	\begin{lemma}\label{lem-compact} 
		Let $\phi$ be a ground state related to \eqref{ell-equ}. Let $\{f_n\}_{n\geq 1}$ be a sequence of functions in $H^1$ satisfying
		\[
		M(f_n) = M(\phi), \quad E(f_n) = E(\phi), \quad \lim_{n\rightarrow \infty} H(f_n) = H(\phi).
		\]
		Then there exist a subsequence still denoted by $\{f_n\}_{n\geq 1}$, a ground state $\tilde{\phi}$ related to \eqref{ell-equ} and a sequence $\{y_n\}_{n\geq1}\subset \R^3$ such that
		\[
		f_n(\cdot+y_n) \rightarrow e^{i\theta} \mu \tilde{\phi}(\mu \cdot) \quad \text{strongly in } H^1
		\]
		for some $\theta \in \R$ and $\mu:=\frac{M(\tilde{\phi})}{M(\phi)}>0$ as $n\rightarrow \infty$.
	\end{lemma}
	
	\begin{proof}
		The proof is based on the concentration-compactness lemma of Lions \cite{Lions}. Since $\{f_n\}_{n\geq 1}$ is a bounded sequence in $H^1$ satisfying $M(f_n) = M(\phi)$ for all $n\geq 1$, it follows from the concentration-compactness lemma of Lions that there exists a subsequence still denoted by $\{f_n\}_{n\geq 1}$ satisfying one of the three possibilities: vanishing, dichotomy and compactness. 
		
		If the vanishing occurs, then $f_n \rightarrow 0$ strongly in $L^r$ for any $2<r<6$ which is not possible due to the fact that
		\begin{align} \label{conver-N}
		-N(f_n) = 2E(f_n)- H(f_n) \rightarrow 2E(\phi) -H(\phi)=-N(\phi) >0
		\end{align}
		as $n\rightarrow \infty,$ where we used the continuity property of the convolution operator with the dipolar kernel $K(x).$ 	\\	
		If the dichotomy occurs, then there exist $\mu \in (0, M(\phi))$ and sequences $\{f^1_n\}_{n\geq 1}, \{f^2_n\}_{n\geq 1}$ bounded in $H^1$ such that
		\begin{equation}\label{comp-cond}
		\begin{aligned}
		\left\{
		\renewcommand*{\arraystretch}{1.2}
		\begin{array}{l}
		\|f_n-f^1_n-f^2_n\|_{L^r} \rightarrow 0 \text{ as } n\rightarrow \infty \text{ for any } 2\leq r <6; \\
		M(f^1_n) \rightarrow \mu, \quad M(f^2_n)\rightarrow M(\phi)-\mu \text{ as } n\rightarrow \infty; \\
		\text{dist}(\supp\{f^1_n\}, \supp\{f^2_n\}) \rightarrow \infty \text{ as } n\rightarrow \infty; \\
		\liminf_{n\rightarrow \infty} H(f_n) - H(f^1_n)-H(f^2_n) \geq 0.
		\end{array}
		\right.
		\end{aligned}
		\end{equation}
		By the Gagliardo-Nirenberg's inequality \eqref{GN-ineq} we have
		\[
		-N(f^1_n) \leq C_{\opt} [H(f^1_n)]^{\frac{3}{2}} [M(f^1_n)]^{\frac{1}{2}} < C_{\opt} [H(f^1_n)]^{\frac{3}{2}} [M(\phi)]^{\frac{1}{2}},
		\]
		 where we let $n$ sufficiently large  and use the second property in \eqref{comp-cond} in order to get the last strict inequality. Similarly we have 
		\[
		-N(f^2_n) < C_{\opt} [H(f^2_n)]^{\frac{3}{2}} [M(\phi)]^{\frac{1}{2}},
		\]
		for $n$ sufficiently large. It follows that
		\begin{align}
		-N(\phi) = \lim_{n\rightarrow \infty} -N(f_n) & = \lim_{n\rightarrow \infty} - N(f^1_n) - N(f^2_n) \label{conver-N-fn}\\
		&< C_{\opt} \lim_{n\rightarrow \infty} \left( [H(f^1_n)]^{\frac{3}{2}} + [H(f^2_n)]^{\frac{3}{2}}\right) [M(\phi)]^{\frac{1}{2}} \nonumber \\
		&\leq C_{\opt} \lim_{n\rightarrow \infty} \left(H(f^1_n) + H(f^2_n)\right)^{\frac{3}{2}} [M(\phi)]^{\frac{1}{2}}  \label{conver-N-fn2} \\
		&\leq C_{\opt} \lim_{n\rightarrow \infty} [H(f_n)]^{\frac{3}{2}} [M(\phi)]^{\frac{1}{2}} \label{conver-N-fn3} \\
		&= C_{\opt} [H(\phi)]^{\frac{3}{2}}[M(\phi)]^{\frac{1}{2}} \nonumber
		\end{align}
		which is a contradiction. To get \eqref{conver-N-fn2} we just used the property $x^\beta+y^\beta\leq(x+y)^\beta$ which holds for any $\beta\geq1$ and any $x,y \in\R$ nonnegative. \eqref{conver-N-fn3} follows by using the last property in \eqref{comp-cond} up to passing to subsequences if necessary.
		\noindent To see \eqref{conver-N-fn}, we proceed as follows. First, we have 
		\begin{align*}
		\left|\|f_n\|^4_{L^4} - \|f_n^1 + f^2_n\|^4_{L^4}\right| &\lesssim \|f_n-f^1_n-f^2_n\|_{L^4} \left( \|f_n\|^3_{L^4} + \|f_n^1+f_n^2\|^3_{L^4}\right) \\
		&\lesssim \|f_n-f^1_n-f^2_n\|_{L^4} \left( \|f_n\|^3_{H^1} + \|f_n^1\|^3_{H^1} + \|f_n^2\|^3_{H^1}\right) \rightarrow 0
		\end{align*}
		as $n\rightarrow \infty$. Since $f^1_n$ and $f^2_n$ have disjoint supports for $n$ large, we infer that
		\begin{align} \label{est-L4}
		\lim_{n\rightarrow \infty} \|f_n\|^4_{L^4} -\|f^1_n\|^4_{L^4} -\|f^2_n\|^4_{L^4} =0.
		\end{align}
		Similarly, using the fact that $\mathcal{K}: f \mapsto K\ast f$ is continuous from $L^2$ into itself,
		\begin{align*}
		\Big| &\int (K\ast |f_n|^2) |f_n|^2 dx - \int (K \ast |f_n^1+f^2_n|^2) |f_n^1+f^2_n|^2 dx\Big| \\
		& =  \int (K\ast |f_n|^2) (|f_n|^2-|f_n^1+f^2_n|^2) dx  + \int (K \ast (|f_n|^2-|f_n^1+f_n^2|^2)) |f_n^1+f_n^2|^2 dx \\
		&\lesssim \|f_n-f_n^1-f^2_n\|_{L^4} \\
		&\mathrel{\phantom{\lesssim}} \times \Big(\|f_n\|^2_{L^4} (\|f_n\|_{L^4}+\|f_n^1+f_n^2\|_{L^4}) + (\|f_n\|_{L^4} + \|f_n^1+f_n^2\|_{L^4}) \|f_n^1+f_n^2\|_{L^4}^2 \Big) \rightarrow 0
		\end{align*} 
		as $n\rightarrow 0$. As $f^1_n$ and $f^2_n$ have disjoint supports for $n$ large, we see that
		\begin{align*}
		\int (K\ast|f_n^1+f^2_n|^2)|f_n^1+f_n^2|^2 dx = \int (K\ast |f_n^1|^2)|f_n^1|^2 dx + \int (K \ast |f_n^2|^2) |f_n^2|^2 dx + o_n(1), 
		\end{align*}
		where $o_n(1) \rightarrow 0$ as $n\rightarrow \infty$. Here we used the fact that
		\begin{equation}\label{cross-term}
		\int (K\ast |f_n^1|^2) |f_n^2|^2 dx = \iint K(x-y) |f_n^1(x)|^2 |f_n^2(y)|^2 dxdy \rightarrow 0
		\end{equation}
		as $n\rightarrow \infty,$ which comes from the boundedness of $\{f^1_n\}_{n\geq 1}$ and $\{f^2_n\}_{n\geq 1}$ in $L^2$ and $K(x-y) \rightarrow 0$ as $|x-y|\rightarrow \infty$. Indeed we have that 
		\begin{equation}\label{cross-to-zero}
		\iint K(x-y) |f_n^1(x)|^2 |f_n^2(y)|^2 dxdy\lesssim \sup_{(x,y) \in \Omega_n} |K(x-y)|\iint |f_n^1(x)|^2 |f_n^2(y)|^2dxdy
		\end{equation}
		where the domain $\Omega_n$ is defined by 
\[
\Omega_n =\{(x,y) \in \R^3\times \R^3 \ : \ x\in \supp\{f_n^1\}, y \in \supp\{f_n^2\}\}.
\]
Hence the convergence \eqref{cross-term} is a consequence of the following:
\[
\sup_{(x,y) \in \Omega_n} |K(x-y)| \|f_n^1\|^2_{L^2} \|f_n^2\|^2_{L^2}
\rightarrow 0
\]
which is implied  by the third property in \eqref{comp-cond} and the decay of $K(x-y)\to0$ as $|x-y|\to\infty.$ 

		Thus we get
		\begin{align} \label{est-N}
		\lim_{n\rightarrow \infty} N(f_n) - N(f^1_n) - N(f^2_n) =0. 
		\end{align}
		By summing up \eqref{est-L4} and \eqref{est-N}, we prove \eqref{conver-N-fn}.	
		
		Therefore we can infer that  compactness must occur. Then there exists a sequence $\{y_n\}_{n\geq 1} \subset \R^3$ such that up to a subsequence, $u_n(\cdot+y_n) \rightarrow f$ strongly in $L^r$ for any $2\leq r<6$ and weakly in $H^1$ with some $f\in H^1$. It follows that
		\begin{align*}
		M(f) = \lim_{n\rightarrow \infty} M(f_n(\cdot+y_n)) = M(\phi), \quad 
		-N(f) = \lim_{n\rightarrow \infty} -N(f_n(\cdot+y_n)) = -N(\phi)
		\end{align*}
		and
		\[
		H(f) \leq \liminf_{n \rightarrow \infty} H(f_n(\cdot+y_n)) = H(\phi).
		\]
		On the other hand, by \eqref{GN-ineq}, we have
		\[
		[H(f)]^{\frac{3}{2}} \geq \frac{-N(f)}{C_{\opt} [M(f)]^{\frac{1}{2}}} = \frac{-N(\phi)}{C_{\opt} [M(\phi)]^{\frac{1}{2}}} = [H(\phi)]^{\frac{3}{2}}
		\]
		hence $H(f) \geq H(\phi)$, so $H(f) = \lim_{n\rightarrow \infty} H(f_n(\cdot+y_n)) = H(\phi)$. Thus, $f_n(\cdot+y_n) \rightarrow f$ strongly in $H^1$ and $f$ is an optimizer for \eqref{GN-ineq}. We claim that there exists $\theta \in \R$ such that $f(x) =e^{i\theta} g(x)$, where $g$ is a non-negative optimizer for \eqref{GN-ineq}. Indeed, since $\|\nabla (|f|)\|_{L^2} \leq \|\nabla f\|_{L^2}$ (see \cite{LL}), it is clear that $W(|f|) \leq W(f)$, where $W$ is as in \eqref{weins-func}. This implies that $|f|$ is also an optimizer for \eqref{GN-ineq} and
		\begin{align} \label{prop-opti-f}
		\|\nabla(|f|)\|_{L^2} = \|\nabla f\|_{L^2}.
		\end{align}
		Set $w(x):=\frac{f(x)}{|f(x)|}$. Since $|w(x)|^2=1$, it follows that $\rea(\overline{w} \nabla w(x)) =0$ and
		\[
		\nabla f(x) = \nabla(|f(x)|) w(x) + |f(x)| \nabla w(x) = w(x) (\nabla(|f(x)|) + |f(x)| \overline{w}(x) \nabla w(x))
		\]
		which implies $|\nabla f(x)|^2 = |\nabla (|f(x)|)|^2+ |f(x)|^2 |\nabla w(x)|^2$ for all $x\in \R^3$. From \eqref{prop-opti-f}, we get
		\[
		\int_{\R^3} |f(x)|^2 |\nabla w(x)|^2 dx =0
		\]
		which shows $|\nabla w(x)|=0$, hence $w(x)$ is a constant, and the claim follows with $g(x) = |f(x)|$. By \cite[Lemma 3.1]{AS}, $g$ is a weak solution to 
		\[
		- a \Delta g + \frac{4}{C_{\opt}} \left(\lambda_1 |g|^2 g+ \lambda_2 (K\ast|g|^2) g\right) + b g =0,
		\]
		where
		\[
		a= 3 \|\nabla g\|_{L^2} \|g\|_{L^2}, \quad b = \|\nabla g\|^3_{L^2} \|g\|_{L^2}^{-1}.
		\]
		By a change of variable $g(x) = \nu \tilde{\phi}(\mu x)$ with $\nu=\sqrt{\frac{b C_{\opt}}{4}}, \mu = \sqrt{\frac{b}{2a}}$, we see that $\tilde{\phi}$ solves \eqref{ell-equ} and $W(g) = W(\tilde{\phi}) =C_{\opt}$. In particular, $\tilde{\phi}$ is a ground state related to \eqref{ell-equ}. Moreover, using
		\[
		H(\phi) M(\phi) = H(f) M(f) = H(g) M(g) = \nu^4 \mu^{-4} H(\tilde{\phi}) M(\tilde{\phi})
		\]
		and \eqref{inde-quant}, we infer that $\nu=\mu$. Since $M(f)=M(\phi)$, we infer that $\mu = \frac{M(\tilde{\phi})}{M(\phi)}$. The proof is complete.		
	\end{proof}
	
	We are now able to show long time dynamics at the mass-energy threshold given in Theorem \ref{theo-dyna-at}. \\
	
	\noindent {\it Proof of Theorem \ref{theo-dyna-at}.} $\bullet$ Let us show the first point. Let $u_0 \in H^1$ satisfy \eqref{cond-ener-at} and \eqref{cond-scat-at}. Since \eqref{cond-ener-at} and \eqref{cond-scat-at} are invariant under the scaling
	\[
	u_0^\rho (x):= \rho u_0( \rho x), \quad \rho>0
	\]
	and taking $\rho= \frac{M(u_0)}{M(\phi)}$, we can assume that 
	\begin{align} \label{cond-EM-at}
	M(u_0) = M(\phi), \quad E(u_0) = E(\phi).
	\end{align}
	The condition \eqref{cond-scat-at} becomes $H(u_0) < H(\phi)$. We claim that
	\begin{align} \label{claim-at-1}
	H(u(t)) < H(\phi)
	\end{align}  
	for all $t\in [0,T^*)$. Assume it is not true, then there exists $t_0 \in [0,T^*)$ such that $H(u(t_0)) = H(\phi)$. By the definition of the energy, we get that
	\[
	- N(u(t_0)) =H(u(t_0))- 2E(u(t_0)) = H(\phi) - 2 E(\phi) = \frac{2}{3} H(\phi) = - N(\phi).
	\]
	This shows that $u(t_0)$ is an optimizer of \eqref{GN-ineq}. Arguing as in the proof of Lemma \ref{lem-compact}, there exists a ground state $\tilde\phi$ related to \eqref{ell-equ} such that
	\[
	u(t_0,x) =  e^{i\theta} \mu \tilde{\phi}(\mu x)
	\]
	for some $\theta \in \R$ and  $\mu >0$. We can rewrite it as $u(t_0,x) = e^{i\mu^2t_0} e^{i(\theta-\mu^2 t_0)} \mu \tilde{\phi}(\mu x)$, hence by uniqueness of the solution to \eqref{dip-NLS}, we have $u(t,x)=e^{i\mu^2 t}e^{i\tilde\theta} \mu \tilde{\phi}(\mu x)$, where $\tilde\theta:=\theta-\mu^2 t_0$. It follows that
	\[
	H(u_0) M(u_0) = H(\tilde{\phi}) M(\tilde{\phi}) = H(\phi) M(\phi)
	\] 
	which contradicts \eqref{cond-scat-at}, and we prove \eqref{claim-at-1}.Combining \eqref{cond-EM-at} and \eqref{claim-at-1}, we prove \eqref{est-solu-at-1}. Moreover, by the blow-up alternative, we have $T^*=\infty.$\\
	
\noindent 	Now we assume in addition that $\lambda_1$ and $\lambda_2$ satisfy \eqref{cond-GW}. We consider two cases:
	
	\noindent \textit{Case 1.} If $\sup_{t\in [0,\infty)} H(u(t)) < H(\phi)$, then there exists $\eta>0$ such that for all $t\in [0,\infty)$,
	\[
	H(u(t)) \leq (1-\eta) H(\phi).
	\]
	This together with \eqref{cond-EM-at} imply
	\begin{align*}
	-N(u(t)) M(u(t)) &\leq C_{\opt} \left(H(u(t)) M(u(t))\right)^{\frac{3}{2}} \\
	&= \frac{2}{3} \frac{\left(H(u(t)) M(u(t))\right)^{\frac{3}{2}}}{\left(H(\phi) M(\phi)\right)^{\frac{1}{2}}} \\
	&\leq \frac{2}{3} (1-\eta)^{\frac{3}{2}} H(\phi) M(\phi) \\
	&=-(1-\eta)^{\frac{3}{2}} N(\phi) M(\phi) 
	\end{align*}
	for all $t\in [0,T^*)$ which shows \eqref{scat-crite}. By Theorem \ref{theo-scat-crite}, the solution scatters in $H^1$ forward in time.
	
	\noindent {\textit{ Case 2.} If $\sup_{t\in[0,\infty)} H(u(t)) = H(\phi)$, then there exists a time sequence $(t_n)_{n \geq 1} \subset [0,\infty)$ such that
	\[
	M(u(t_n)) = M(\phi), \quad E(u(t_n))= E(\phi), \quad \lim_{n\rightarrow \infty} H(u(t_n)) = H(\phi).
	\]
	Note that $t_n$ must tend to infinity. Indeed, if it does not hold, then up to a subsequence $t_n \rightarrow t_0$ as $n\rightarrow \infty$. Since $u(t_n) \rightarrow u(t_0)$ strongly in $H^1$, we see that $u(t_0)$ is an optimizer for \eqref{GN-ineq}. Arguing as above, we have a contradiction. We now apply Lemma \ref{lem-compact} with $f_n= u(t_n)$ to get: (up to a subsequence) there exist a ground state $\tilde{\phi}$ related to \eqref{ell-equ} and a sequence $\{y_n\}_{n\geq 1} \subset \R^3$ such that
	\[
	u(t_n, \cdot+y_n) \rightarrow e^{i\theta} \mu \tilde{\phi}(\mu \cdot) \quad \text{ strongly in } H^1
	\]
	for some $\theta \in \R$ and $\mu>0$ as $n\rightarrow \infty$. This finishes the first part of Theorem \ref{theo-dyna-at}.
	
	$\bullet$ We continue with the proof of the second point. Let $u_0 \in H^1$ satisfy \eqref{cond-ener-at} and \eqref{cond-at}. By scaling, we can assume
	\[
	M(u_0) = M(\phi), \quad E(u_0) = E(\phi), \quad H(u_0) = H(\phi).
	\]
	This shows that $u_0$ is an optimizer for \eqref{GN-ineq}. This shows that $u_0(x) = e^{i\theta} \mu \tilde{\phi}(\mu x)$ for some $\theta \in \R,$  $\mu>0$ and $\tilde \phi$ a ground state related to \eqref{ell-equ}. By the uniqueness of solutions to \eqref{dip-NLS}, we conclude that $u(t,x) = e^{i\mu^2t} e^{i\tilde{\theta}} \mu \tilde{\phi}(\mu x)$ for some $\tilde{\theta}\in \R$.
	
	$\bullet$ Finally, we consider the third point. Let $u_0 \in H^1$ satisfy \eqref{cond-ener-at} and \eqref{cond-blow-at}. By scaling argument, we can assume that \eqref{cond-EM-at} holds. Then, \eqref{cond-blow-at} becomes $H(u_0)> H(\phi)$. By the same argument as in the proof of \eqref{claim-at-1}, we prove that
	\begin{align} \label{claim-at-2}
	H(u(t))> H(\phi)
	\end{align}
	for all $t\in [0,T^*)$. If $T^*<\infty$, then we are done. Otherwise, if $T^*=\infty$, we consider two cases.
	
	\noindent \textit{Case 1.} If $\sup_{t\in [0,\infty)} H(u(t)) > H(\phi)$, then there exists $\eta>0$ such that for all $t\in [0,\infty)$,
	\[
	H(u(t)) \geq (1+\eta) H(\phi).
	\]
	From a simple computation by \eqref{def:G} and the definition of the energy, it follows that
	\begin{align*}
	G(u(t)) M(u(t)) &= 3 E(u(t)) M(u(t)) - \frac{1}{2} H(u(t)) M(u(t)) \\
	&= 3 E(u_0) M(u_0) - \frac{1}{2} H(u(t)) M(u_0) \\
	&\leq 3 E(\phi) M(\phi) - \frac{1}{2} (1+\eta) H(\phi) M(\phi) \\
	&= -\frac{\eta}{2}H(\phi) M(\phi)
	\end{align*}
	for all $t\in [0,\infty),$ where in the last equality we used \eqref{inde-quant-proof}. By Theorem \ref{theo-blow-crite}, there exists a time sequence $t_n \rightarrow \infty$ such that $\|u(t_n)\|_{H^1} \rightarrow \infty$ as $n\rightarrow \infty$.
	
	\noindent \textit{Case 2.} If $\sup_{t\in [0,\infty)} H(u(t)) = H(\phi)$, we can argue as above to have: there exist a time sequence $t_n \rightarrow \infty$, a ground state $\tilde{\phi}$ related to \eqref{ell-equ} and a sequence $\{y_n\}_{n\geq 1} \subset \R^3$ such that 
	\[
	u(t_n, \cdot +y_n) \rightarrow e^{i\theta} \mu \tilde{\phi}(\mu \cdot) \quad \text{ strongly in } H^1
	\]
	for some $\theta \in \R$ and $\mu>0$ as $n\rightarrow \infty$. The proof of Theorem \ref{theo-dyna-at} is now complete.

\section*{Acknowledgements}
The authors are grateful to the anonymous referees for their comments and remarks, which improved the presentation of the paper. 
	V. D. D. was supported in part by the Labex CEMPI (ANR-11-LABX-0007-01). He would like to express his deep gratitude to his wife - Uyen Cong for her encouragement and support.

\end{document}